\documentclass[reqno,twoside]{amsart} 
\usepackage{verbatim}
\usepackage{amsmath,amsthm,amssymb,amsfonts,mathrsfs,color}
\usepackage{latexsym,esint}
\theoremstyle{plain}
\begingroup
\newtheorem{thm}{Theorem}[section]
\newtheorem{lem}[thm]{Lemma}
\newtheorem{prop}[thm]{Proposition}

\endgroup

\theoremstyle{definition}
\begingroup
\newtheorem{rem}[thm]{Remark}
\endgroup

\numberwithin{equation}{section}

\makeatletter
\def\rightharpoonupfill@{\arrowfill@\relbar\relbar\rightharpoonup}
\newcommand{\xrightharpoonup}[2][]{\ext@arrow 0359\rightharpoonupfill@{#1}{#2}}
\makeatother

\newcommand{\ds}{\displaystyle}
\usepackage{yhmath}
\newcommand{\CC}{\mathbb C} 
\newcommand{\N}{\mathbb N} 
\newcommand{\R}{\mathbb R} 
\newcommand{\Ms}{{\mathbb M}^{2{\times}2}_{\rm sym}}
\newcommand{\M}{{\mathbb M}^{2\times 2}}

\newcommand{\diam}{{\rm diam}}
\renewcommand{\div}{{\rm div}}
\renewcommand{\Cap}{{\rm Cap}}
\newcommand{\wto}{\rightharpoonup}
\newcommand{\e}{\varepsilon}
\newcommand{\ol}{\overline}

\newcommand{\D}{{\mathcal D}}

\newcommand{\LL}{{\mathcal L}}
\newcommand{\HH}{{\mathcal H}}
\newcommand{\MM}{{\mathcal M}}
\newcommand{\C}{{\mathcal C}}
\newcommand{\F}{{\mathcal F}}
\newcommand{\G}{{\mathcal G}}
\newcommand{\K}{{\mathcal K}}

\let\O=\Omega

\begin{document}
 
\title{Energy release rate for non smooth cracks in planar elasticity}
\author[J.-F. Babadjian]{Jean-Fran\c cois Babadjian}
\author[A. Chambolle]{Antonin Chambolle}
\author[A. Lemenant]{Antoine Lemenant}

\address[J.-F. Babadjian]{Universit\'e Pierre et Marie Curie -- Paris 6, CNRS, UMR 7598 Laboratoire Jacques-Louis Lions, Paris, F-75005, France}
\email{jean-francois.babadjian@upmc.fr}

\address[A. Chambolle]{CMAP, UMR 7641, Ecole Polytechnique, CNRS, 91128 Palaiseau, France}
\email{antonin.chambolle@polytechnique.fr}

\address[A. Lemenant]{Universit\'e Paris Diderot -- Paris 7, CNRS, UMR 7598 Laboratoire Jacques-Louis Lions, Paris, F-75005, France}
\email{lemenant@ljll.univ-paris-diderot.fr}

\date{\today}

\begin{abstract}This paper is devoted to the characterization  of the energy release rate of a crack which is merely closed, connected, and with density $1/2$  at the tip.  First, the blow-up limit of the displacement is analyzed, and the convergence to the corresponding positively $1/2$-homogenous function in the cracked plane is established. Then, the energy release rate is obtained as the derivative of the elastic energy with respect to an infinitesimal additional crack increment.  %Our results extend in the vectorial case and with a different proof, some scalar results obtained in \cite{CL}, and recover some results of \cite{CFM} in which only segments where admissible cracks.
\end{abstract}

\maketitle

\tableofcontents

\section{Introduction}

\noindent Griffith theory \cite{G} is a model explaining the quasi-static crack growth in elastic bodies under the assumption that the crack set is preassigned. In a two-dimensional setting, let us denote by $\O\subset \R^2$ the reference configuration of a linearly elastic body allowing for cracks inside $\hat \Gamma$. To fix the ideas, provided the evolution is sufficiently smooth, that $\hat \Gamma$ is a simple curve, and that the evolution is growing only in one direction, then the crack is completely characterized by the position of its tip, and thus by its arc length. Denoting by $\Gamma(\ell)$ the crack of length $\ell$ inside $\hat \Gamma$, the elastic energy associated to a given kinematically admissible displacement $u : \O \setminus \Gamma(\ell) \to \R^2$ satisfying $u=\psi(t)$ on $\partial \O \setminus \Gamma(\ell)$, is given by
$$E(t;u,\ell):=\frac{1}{2} \int_{\O \setminus \Gamma(\ell)} \CC e(u):e(u)\, dx,$$
where $\CC$ is the fourth order Hooke's tensor, and $\psi(t):\partial \O \to \R^2$ is a prescribed boundary datum depending on time, which is the driving mechanism of the process. If the evolution is slow enough, it is reasonable to neglect inertia and viscous effects so that the quasi-static assumption becomes relevant: at each time $t$, the body is in elastic equilibrium. It enables one to define the potential energy as
$$\mathcal P(t,\ell):=E(t;u(t,\ell),\ell)=\min E(t; \cdot,\ell),$$
where the minimum is computed over all kinematically admissible displacements at time $t$. Therefore, given a cracking state, the quasi-static assumption permits to find the displacement. In order to get the crack itself (or equivalently its length), Griffith introduced a criterion whose fundamental ingredient is the {\it energy release rate}. It is defined as the variation of potential energy along an infinitesimal crack increment, or in other words, the quantity of released potential energy with respect to a small crack increment. More precisely, it is given by
$$G(t,\ell):=-\frac{\partial \mathcal P}{\partial \ell}(t,\ell)$$
provided the previous expression makes sense. From a thermodynamical point of view, the energy release rate is nothing but the thermodynamic force associated to the crack length (the natural internal variable modeling the dissipative effect of fracture). Griffith criterion is summarized into the three following items: for each $t>0$
\begin{itemize}
\item[(i)] $G(t,\ell(t)) \leq G_c$, where $G_c>0$ is a characteristic material constant referred to as the toughness of the body;
\item[(ii)] $\dot \ell(t) \geq 0$;
\item[(iii)] $(G(t,\ell(t)) - G_c) \dot \ell(t)=0$.
\end{itemize}
Item (i) is a threshold criterion which stipulates that the energy release rate cannot exceed the critical value $G_c$. Item (ii) is an irreversibility criterion which  ensures that the crack can only grow. The third and last item is a compatibility condition between (i) and (ii):  it states that a crack will grow if and only if the energy release rate constraint is saturated.

In \cite{FM} (see also \cite{BFM}), it has been observed that Griffith is nothing but the necessary first order optimality condition of a variational model. More precisely, if for every $t>0$, $(u(t),\ell(t))$ satisfies:
\begin{itemize}
\item[(i)] {\it Unilateral minimality:} for any $\hat \ell \geq \ell(t)$, and any $v : \O \setminus \Gamma(\hat \ell) \to \R^2$ satisfying $v=\psi(t)$ on $\partial \O \setminus \Gamma(\hat \ell)$, then
$$\mathcal E(t):=\frac12 \int_{\O \setminus \Gamma(\ell(t))}Ê\CC e(u(t)):e(u(t))\, dx + G_c \, \ell(t) \leq \frac12 \int_{\O \setminus \Gamma(\hat \ell)}Ê\CC e(v):e(v)\, dx + G_c \, \hat \ell;$$
\item[(ii)] {\it Irreversibility:} $\dot \ell(t) \geq 0$;
\item[(iii)] {\it Energy balance:}
$$\dot{\mathcal E}(t)=\int_{\partial \O \setminus \Gamma(\ell(t))} (\CC e(u(t)) \nu) \cdot \dot \psi(t)\, d\HH^1,$$
\end{itemize}
then $(u(t),\ell(t))$ is a solution of Griffith' model. In the previous expression, $\HH^1$ denotes the $1$-dimensional Hausdorff measure. The energy balance is nothing but a reformulation of the second law of thermodynamics which asserts the non-negativity of the mechanical dissipation. It states that the temporal variation of the total energy (the sum of the elastic and surface energies) is compensated by the power of external forces, which in our case reduces to the stress $(\CC e(u(t)) \nu$ acting on $\partial \O \setminus \Gamma(\ell(t))$ and generated by the boundary displacement $\psi(t)$. This new formulation relies on the constrained minimization of the total energy of Mumford-Shah type
$$\mathcal E(u,\Gamma):=\frac12 \int_{\O\setminus \Gamma} \CC e(u):e(u)\, dx +G_c \HH^1(\Gamma)$$
which put in competition a bulk (elastic) energy and a surface (Griffith) energy. One of the main interests is that it makes it possible to get rid of the assumption of the {\it a priori} knowledge of the crack path. Following \cite{FM}, a quasi-static evolution is defined as a mapping $t \mapsto (u(t),\Gamma(t))$ satisfying
\begin{itemize}
\item[(i)] {\it Unilateral minimality:} for any $\overline \O \supset \hat \Gamma \supset \Gamma(t)$, and any $v : \O \setminus \hat \Gamma \to \R^2$ satisfying $v=\psi(t)$ on $\partial \O \setminus \hat \Gamma$, then
$$\mathcal E(u(t),\Gamma(t)) \leq \mathcal E(v,\hat\Gamma);$$
\item[(ii)] {\it Irreversibility:} $\Gamma(s) \subset \Gamma(t)$ for every $s \leq t$;
\item[(iii)] {\it Energy balance:}
$$\mathcal E(u(t),\Gamma(t))=\mathcal E(u(0),\Gamma(0))+\int_0^t \int_{\O \setminus \Gamma(s)} \CC e(u(s))  :e( \dot \psi(s))\, dx\, ds.$$
\end{itemize}
An existence result for this model has been given in \cite{C} (see also \cite{DMT,FL,DMFT} in other contexts) for cracks belonging to the class of compact and connected subsets of $\overline \O$. The main reason of this assumption was to ensure the lower semicontinuity of the Mumford-Shah type functional $(u,\Gamma) \mapsto \mathcal E(u,\Gamma)$ with respect to a reasonable notion of convergence. The lower semicontinuity of the surface energy with respect to the Hausdorff convergence of cracks is a consequence of Go\l ab's Theorem (see \cite{Falconer}), while the continuity of the bulk energy is a consequence of continuity results of the Neumann problem with respect to the Hausdorff convergence of the boundary (see \cite{BV,CD}) together with a density result \cite{C}. In any cases, all these results only hold in dimension 2 and in the class of compact and connected sets. 

If one is interested into fine qualitative results such as crack initiation (see \cite{CGP}) of kinking (see \cite{CFM}) it becomes necessary to understand the nature of the singularity at the crack tip. Therefore one should be able to make rigorous a suitable notion energy release rate. The first proof of the differentiable character of the potential energy with respect to the crack length has been given in \cite{DD} (see also \cite{KM,NO,LT}). The generalized variational setting described above, a mathematical justification of the notions of energy release rate  for any incremental crack attached to a given initial crack has been in \cite{CFM} in the case where the crack is straight in a small neighborhood of its tip. In the footstep of that work, we attempt here weaken the regularity assumption on the initial crack, which is merely closed, connected, with density 1/2 at the origin (that imply to blow up as a segment at the origin, up to rotations). %We will base our study upon the variational formulation of crack propagation introduced by {\sc Francfort \& Marigo} (see \cite{BFM}), using a functional of Mumford-Shah type, and restrict ourselves only to the two-dimensional case. 

\subsection{Main Results}

Our main results are contained in Theorem \ref{thm:blowup} and  Theorem \ref{main2} respectively in Section \ref{sec:blow-up} and Section \ref{sec:err}.

\subsubsection{First Result.} The first main result Theorem \ref{thm:blowup} is a purely P.D.E. result.  We  analyze the blow-up limit of the optimal displacement at the tip of the given initial crack. We prove that for some suitable subsequence, the blow-up limit converges to the classical crack-tip function in the complement of a half-line, {\it i.e.} of the form  
\begin{equation}\label{cracktip0007}
\kappa_1\phi_1 + \kappa_2\phi_2,
\end{equation}
for some constants $\kappa_1$ and $\kappa_2 \in \R$, while $\phi_1$ and $\phi_2$ are positively 1/2-homogenous functions which are explicitly given by \eqref{defPHI1} and \eqref{defPHI2} below.

This part can be seen as a partial generalization in planar elasticity of what was previously done in the anti-plane case \cite{CL}.  Mathematically speaking,  the corresponding function to be studied is  now a vectorial function satisfying a Lam\'e type system, instead of  being simply a scalar valued harmonic function. One of the key obstacle in the vectorial case is that  no monotonicity property is known for such a problem, which leads to a slightly weaker result than in the scalar case: the convergence of the blow-up sequence only holds up to subsequences, and nothing is known for the whole sequence. Consequently, the constants $\kappa_1$ and $\kappa_2$  in \eqref{cracktip0007} {\it a priori}  depend on this particular subsequence. As a matter of fact, this prevents us to define properly the stress intensity factor analogously to what was proposed in \cite{CL}. On the other hand, we believe that the techniques employed in the proof and the results on their own are already interesting. In addition, the absence of monotonicity is not the only difference with the scalar case, which led us to find a new proof relying on a duality approach via the so-called Airy function in order to bypass some technical problems.

Another substantial difference with the scalar case appears while studying homogeneous solutions of the planar Lam\'e system in the complement of a half-line, which is crucial in the understanding of blow-up solutions at the crack tip. For harmonic functions it is rather easy to decompose any solutions as a sum of spherical-harmonics directly by writing the operator in polar coordinates, and identify the degree of homogeneity of each term with the corresponding eigenvalue of the Dirichlet-Laplace-Beltrami operator on the circle minus a point. For the Lam\'e system, or alternatively for the biharmonic equation,  a similar naive approach cannot work. The appropriate eigenvalue problem on the circle have a more complicate nature, and analogous results rely on an abstract theory developed first by Kondrat'ev which rests on pencil operators, weighted Sobolev spaces, the Fredholm alternative, and calculus of residues.  We used this technology in the proof of   Proposition \ref{eq:3/2homo} for which we could not find a more elementary argument.

\subsubsection{Second result.} The second main result Theorem \ref{main2} concerns the energy release rate of an incremental crack $\Gamma$, which is roughly speaking the derivative of the elastic energy with respect to the crack increment %defined as being the limit, when $\varepsilon$ goes to $0$, of the bulk energy divided by $\varepsilon$ associated  to the incremental crack $\varepsilon \Gamma$ 
(see  \eqref{ERRate} for the precise definition). We prove that the value of this limit is realized as an explicit minimization problem in the cracked-plane $\R^2\setminus \big((-\infty, 0]\times \{0\}\big)$. One can find a similar statement in \cite[Theorem 3.1]{CFM}, but with the additional assumption that the initial crack is a line segment close to the origin. We remove here this hypothesis, establishing  the same result for any initial crack which is closed, connected and admits a line segment as blow-up limit at the origin. The starting point for this generalization is the knowledge of the blow-up limit at the origin for displacement associated to a general initial crack, namely our first result Theorem \ref{thm:blowup}. Since this result holds only up to subsequences, the same restriction appears in the statement of Theorem \ref{thm:blowup} as well.  

Therewith, it should be mentioned that Theorem \ref{main2} is new even for the scalar case, for which the conclusion is even more accurate. Indeed in this case, the monotonicity formula of \cite{CL} ensures that the convergence holds  for the whole sequence and not only for a subsequence.
 
\bigskip
 
The paper is organized as follows: after introducing the main notation in section \ref{sec:math}, we describe precisely the mechanical model in section \ref{sec:mech}. Section \ref{sec:prel} is devoted to establish technical results related to the existence of the harmonic conjugate and the Airy function associated to the displacement in a neighborhood of the crack tip. In section \ref{sec:bounds}, we prove lower and upper bounds of the energy release rate. The blow-up analysis of the displacement around the crack tip is the object of section \ref{sec:blow-up}. Section \ref{sec:err} is devoted to give a formula for the energy release rate as a global minimization problem. Finally, we state in an appendix a Poincar\'e inequality in a cracked annulus, and shortly review Kondrat'ev theory of elliptic regularity vs singularity inside corner domains.
 
\section{Mathematical preliminaries} \label{sec:math}

\subsection{General notation} 

The Lebesgue measure in $\R^n$ is denoted by $\LL^n$, and the $k$-dimensional Hausdorff measure by $\HH^k$. If $E$ is a measurable set, we will sometimes write $|E|$ instead of $\LL^n(E)$. If $a$ and $b \in \R^n$, we write $a \cdot b=\sum_{i=1}^n a_i b_i$ for the Euclidean scalar product, and we denote the norm by $|a|=\sqrt{a \cdot a}$. The open ball of center $x$ and radius $\varrho$ is denoted by $B_\varrho(x)$. If $x=0$, we simply write $B_\varrho$ instead of $B_\varrho(0)$.

\medskip

We write $\mathbb M^{n \times n}$ for the set of real $n \times n$ matrices, and $\mathbb M^{n \times n}_{\rm sym}$ for that of all real symmetric $n \times n$ matrices. Given a matrix $A \in \mathbb M^{n \times n}$, we let $|A|:=\sqrt{{\rm tr}(A A^T)}$ ($A^T$ is the transpose of $A$, and ${\rm tr }A$ is its trace) which defines the usual Euclidean norm over $\mathbb M^{n \times n}$. We recall that for any two vectors $a$ and $b \in \R^n$, $a \otimes b \in \mathbb M^{n \times n}$ stands for the tensor product, {\it i.e.}, $(a \otimes b)_{ij}=a_i b_j$ for all $1 \leq i,j \leq n$, and $a \odot b:=\frac12 (a \otimes b + b \otimes a) \in \mathbb M^{n \times n}_{\rm sym}$ denotes  the symmetric tensor product. 

\medskip

Given an open subset $U$ of $\R^n$, we denote by $\mathcal M(U)$ the space of all real valued Radon measures with finite total variation. We use standard notation for Lebesgues spaces $L^p(U)$ and Sobolev spaces $W^{k,p}(U)$ or $H^k(U):=W^{k,2}(U)$. If $\Gamma$ is a closed subset of $\overline U$, we denote by $H^k_{0,\Gamma}(U)$ the closure of $\C^\infty_c(\overline U \setminus \Gamma)$ in $H^k(U)$. In particular, if $\Gamma=\partial U$, then $H^k_{0,\partial U}(U)=H^k_0(U)$. %When $m=1$ we simply write $\mathcal M(U)$, $L^p(U)$, $H^k(U)$ and $W^{k,p}(U)$ instead of $\mathcal M(U;\R)$, $L^p(U;\R)$, $H^k(U;\R)$ and $W^{k,p}(U;\R)$. 

\subsection{Capacities}

In the sequel, we will use the notion of capacity for which we refer to \cite{AH,HP}. We just recall the definition and several facts. The $(k,2)$-capacity of a compact set $K \subset \R^n$ is defined by
$$\Cap_{k,2}(K):=\inf \left\{ \|\varphi\|_{H^k(\R^n)} : \varphi \in \C^\infty_c(\R^n),\; \varphi \geq 1 \text{ on } K\right\}.$$
This definition is then extended to open sets $A \subset \R^n$ by
$$\Cap_{k,2}(A):= \sup\big\{ \Cap_{k,2}(K) : K \subset A, \; K \text{ compact}\big\},$$
and to arbitrary sets $E \subset \R^n$ by
$$\Cap_{k,2}(E):= \inf\big\{ \Cap_{k,2}(A) : E \subset A, \; A \text{ open}\big\}.$$
One of the interests of capacity is that it enables one to give an accurate sense to the pointwise value of Sobolev functions. More precisely, if $u \in H^k(\R^n)$ then $u$ is $(k,2)$-quasicontinuous which means that for each $\e>0$, there exists an open set $A_\e \subset \R^n$ such that $\Cap_{k,2}(\R^n \setminus A_\e)<\e$ and $u$ is continuous in $A_\e$ (see \cite[Section 6.1]{AH}). In addition, if $U$ is an open subset of $\R^n$, then $uÊ\in H^k_0(U)$ if and only if for all multi-index $\alpha \in \N^n$ with length $|\alpha|\leq k$, $\partial^\alpha u$ has a $(k-|\alpha|,2)$-quasicontinuous representative that vanishes ${\rm Cap}_{k-|\alpha|,2}$-quasi everywhere on $\partial U$, {\it i.e.} outside a set of zero ${\rm Cap}_{k-|\alpha|,2}$-capacity (see \cite[Theorem 9.1.3]{AH}). In the sequel, we will only be interested to the cases $k=1$ or $k=2$ in dimension $n=2$.

%\begin{lem}
%Let $U$ and $V$ be a open subsets of $\R^n$ such that $\overline U \subset V$, $\Gamma$ be a closed subset of $\overline U$, and $u \in H^k(V)$. Then $uÊ\in H^k_{0,\Gamma}(U)$ if and only if for all multi-index $\alpha \in \N^n$ with length $|\alpha|\leq k$, $D^\alpha u$ has a $(k-|\alpha|,2)$-quasicontinuous representative that vanishes ${\rm Cap}_{k-|\alpha|,2}$-q.e. on $\Gamma$.
%\end{lem}
%
%\begin{proof}
%Consider a cut-off function $\zeta \in \C^\infty_c(V;[0,1])$ such that $\zeta =1$ on $U$, and define $v=\zeta u \in H^k_0(V)$. By the previous property, $v \in H^k_0(V \setminus \Gamma)$ if and only if for all multi-index $\alpha \in \N^n$ with length $|\alpha|\leq k$, $D^\alpha v$ has a $(k-|\alpha|,2)$-quasicontinuous representative that vanishes ${\rm Cap}_{k-|\alpha|,2}$-q.e. in $\partial (VÊ\setminus \Gamma)=\partial V \cup \Gamma$.
%\end{proof}

\subsection{Kondrat'ev spaces}\label{sectionKond}

%It will be convenient for us to follow  the notation and statements contained in the book 
Following \cite[Section 6.1]{kmross}, if $C$ is an open cone of $\R^n$ with vertex at the origin, we define for any $\beta \in\R$ and $\ell \geq 0$ the weighted Sobolev space
%that we briefly recall here in the simple case of the bi-Laplace operator in a cracked plane, i.e. $G:=\R^2\setminus \Gamma$.  For $\beta \in\R$ and $\ell \geq 0$ 
$V_\beta^\ell(C)$ by the closure of $\C_c^\infty(\overline C \setminus \{0\})$ with respect to the norm
$$\| u\|_{V_\beta^\ell(C)}:= \Big(\int_C \sum_{|\alpha| \leq \ell} |x|^{2(\beta-\ell+|\alpha|)}|\partial^\alpha u(x)|^2 dx\Big)^{\frac{1}{2}}.$$
It will also be useful to introduce the spaces $V_\beta^{\ell}(C)$ for  $\ell<0$, which is defined as the dual space of $V_{-\beta}^{-\ell}(C)$, endowed with the usual dual norm. 

%Spaces of negative order was not introduced in the original paper of Kondrat'ev \cite{kond67}, but his results naturally extend, as stated in \cite{kmross}.

%For $\ell\geq 1$ it is easy to define the space of traces for function in $V_\beta^\ell(G)$ and identify it as a weighted space as well, precisely the space $V_{\beta}^{\ell-\frac{1}{2}}(\partial G)$. But since here we focus only on homogeneous Dirichlet conditions, we shall not need to take care too much about  trace spaces. 

Observe that  when $\ell\geq 0$ then $u \in V_\beta^\ell(C)$ if and only if the function  $x \mapsto |x|^{\beta-\ell+|\alpha|}\partial^\alpha u(x) \in L^2(C)$ for all $|\alpha|\leq \ell$. If one is interested in homogeneous functions, it turns out that the parameter $\beta$ plays a different role regarding to the integrability at the origin or at infinity. To fix the ideas, one can check that in dimension 2, a  function of the form $x \mapsto |x|^\gamma f(x/|x|)$ around the origin and with compact support belongs to $ V_\beta^\ell(C)$ for every $\beta<1-\gamma$. On the other hand, a function having this behavior at infinity and vanishing around the origin will belong to a space $ V_\beta^\ell(C)$ for every $\beta>1-\gamma$. For instance if $\gamma=3/2$, then the corresponding space of critical exponent would be that with $\beta=-1/2$. 

\subsection{Fonctions with Lebesgue deformation}

Given a vector field (distribution) $u : U \to \R^n$, the symmetrized gradient of $u$ is denoted by 
$$e(u):=\frac{\nabla u + \nabla u^T}{2}.$$
In linearized elasticity, $u$ stands for the displacement, while $e(u)$ is the elastic strain. The elastic energy of a body is given by a quadratic form of $e(u)$ so that it is natural to consider displacements such that $e(u) \in L^2(U;\mathbb M^{n \times n}_{\rm sym})$. If $U$ has Lipschitz boundary, it is well known that $u$ actually belongs to $H^1(U;\R^n)$ as a consequence of Korn's inequality (see {\it e.g.} \cite{Ciarlet,T1}). % and that Korn's Inequality holds: there exists a constant $c_U>0$ such that
%$$\|u\|_{H^1(U)} \leq c_U \left( \|u\|_{L^2(U)} + \|e(u)\|_{L^2(U)}\right).$$
However, when $U$ is not smooth,  we can only assert that  $u \in L^2_{\rm loc}(U;\R^n)$. This motivates the following definition of the space of Lebesgue deformations:
$$LD(U):=\{ u \in L^2_{\rm loc}(U;\R^n) : e(u) \in L^2(U;\mathbb M^{n \times n}_{\rm sym})\}.$$
If $U$ is connected and $u$ is a distribution with $e(u)=0$, then necessarily it is a rigid movement, {\it i.e.}  $u(x)=Ax+b$ for all $x \in U$, for some skew-symetric matrix $A \in \mathbb M^{n \times n}$ and some vector $b \in \R^n$. If, in addition, $\partial U$ is locally contained inside a finite union of Lipschitz graphs, the following Poincar\'e-Korn inequality holds: there exists a constant $c_U>0$ and a rigid movement $r_U$ such that
\begin{equation}\label{poincare-korn}
\|u-r_U\|_{L^2(U)}Ê\leq c_U \|e(u)\|_{L^2(U)}, \quad \text{for all }u \in LD(U).
\end{equation}
According to \cite[Theorem 5.2, Example 5.3]{AMR}, it is possible to make $r_U$ more explicit in the following way: consider a measurable subset $E$ of $U$ with $|E|>0$, then one can take 
$$r_U(x):=\frac{1}{|E|}\int_E u(y)\, dy + \left( \frac{1}{|E|}\int_E\frac{\nabla u(y) - \nabla u(y)^T}{2}\, dy \right)\left(x-\frac{1}{|E|}\int_E y\, dy \right),$$
provided the constant $c_U$ in \eqref{poincare-korn} also depends on $E$.

\subsection{Hausdorff convergence of compact sets}

Let $K_1$ and $K_2$ be compact subsets of a common compact set $K \subset \R^n$. The Hausdorff distance between $K_1$ and $K_2$ is given by
$$
d_\HH(K_1,K_2):=\max\left\{ \sup_{x \in K_1}{\rm dist}(x,K_2), \sup_{y \in K_2}{\rm dist}(y,K_1)\right\}.
$$
We say that a sequence $(K_n)$ of compact subsets of $K$ converges in the Hausdorff distance to the compact set $K_\infty$ if 
$d_\HH(K_n,K_\infty) \to 0$. The Hausdorff convergence of compact sets turns out to be equivalent to the convergence in the sense of Kuratowski. Indeed $K_n \to K_\infty$ in the Hausdorff metric if and only if both following properties hold:
\begin{itemize}
\item[a)] any $x \in K_\infty$ is the limit of a sequence $(x_n)$ with $x_n \in K_n$;
\item[b)] if $\forall n,  \; x_n \in K_n$, any limit point of $(x_n)$ belongs to $K_\infty$.
\end{itemize}
Finally let us recall Blaschke's selection principle which asserts that 
from any sequence $(K_n)$ of compact subsets of $K$, one can extract a subsequence converging in the Hausdorff distance.

\section{Description of the model}\label{sec:mech}

\noindent {\bf Reference configuration.} We consider a homogeneous isotropic linearly elastic body occupying $\O$ in its reference configuration, a bounded and connected open subset of $\R^2$ with Lipschitz boundary. We suppose that the stress $\sigma\in \Ms$ is related to the strain $e\in \Ms$ thanks to Hooke's law
$$\sigma=\CC e=\lambda ({\rm tr} e) I+2\mu e,$$
where $\lambda>0$ and $\mu>0$ are the Lam\'e coefficients, and $I$ is the identity matrix. This expression can be inverted into
\begin{equation}\label{eq:Hooke-1}
e=\CC^{-1}\sigma=\frac{1+\nu}{E}\sigma -\frac{\nu}{E}({\rm tr} \sigma) I,
\end{equation}
where $E:=\frac{\mu(3\lambda+2\mu)}{\lambda+\mu}$ is the Young modulus and $\nu:=\frac{\lambda}{2(\lambda+\mu)}$ is the Poisson coefficient.

\medskip

\noindent  {\bf External loads.} We suppose that the body is only subjected to a soft device loading, that is, to a prescribed displacement $\psi \in H^{1/2}(\partial \O;\R^2)$ acting on the entire boundary. 

\medskip

\noindent  {\bf Admissible cracks. }We further assume that the body can undergo cracks which belong to the admissible class
$$\K(\overline \O):=\{ \Gamma \subset \overline \O \text{ compact, connected, }0 \in \Gamma \text{ and }\HH^1(\Gamma)<\infty\}.$$

\medskip

\noindent  {\bf Admissible displacements.} For a given crack $\Gamma \in \K(\overline \O)$, we define the space of admissible displacement by
$$LD(\O \setminus \Gamma) :=\{ u \in L^2_{\rm loc}(\O \setminus \Gamma;\R^2) : e(u) \in L^2(\O\setminus \Gamma;\Ms)\}.$$
If $B$ is a ball with $\overline B \cap \Gamma= \emptyset$, then $\O \cap B$ has Lipschitz boundary so that Korn's inequality ensures that $u\in H^1(\O \cap B;\R^2)$. As a consequence, the trace of $u$ is well defined on $\partial \O \cap B$. Since this property holds for any ball as above, then the trace of $u$ is well defined on $\partial \O \setminus \Gamma$. 

\medskip

\noindent  {\bf Initial data. }We consider an initial crack $\Gamma_0 \in \K(\overline \O)$ satisfying furthermore
\begin{equation}\label{eq:dens1/2}
\lim_{\varrho \to 0} \frac{\HH^1(\Gamma_0 \cap B_\varrho)}{2\varrho}=\frac12,
\end{equation}
and an associated displacement $u_0 \in LD(\O\setminus \Gamma_0)$ given as a solution of the minimization problem
\begin{equation}\label{minprob1}
\min\left\{\frac12 \int_{\O \setminus \Gamma_0} \CC e(v):e(v)\, dx : v \in LD(\O\setminus \Gamma_0), \; v=\psi \text{ on }\partial\O \setminus \Gamma_0  \right\}. \end{equation}
Note that $u_0$ is unique up to an additive rigid movement in each connected component of $\O \setminus \Gamma_0$ disjoint from $\partial \O \setminus \Gamma_0$. However, the stress, which is given by Hooke's law
\begin{eqnarray}
\sigma_0 :=\CC e(u_0) \in L^2(\O \setminus \Gamma_0 ; \Ms) \label{defsigma0}
\end{eqnarray}
is unique and it satisfies the variational formulation 
\begin{equation}\label{eq:var-form}
\int_{\O \setminus \Gamma_0} \sigma_0 : e(v)\, dx = 0
\end{equation}
for any $v \in LD(\O\setminus \Gamma_0)$ such that $ v=0$ on $\partial\O \setminus \Gamma_0$. Note that standard results on elliptic regularity (see {\it e.g.} \cite[Theorem 6.3.6]{Ciarlet}) ensure that $u_0 \in \C^\infty(\O \setminus \Gamma_0;\R^2)$. 

\medskip

\noindent {\bf Energy release rate.} To define the energy release rate, let us consider a crack increment $\Gamma_0 \cup \Gamma$, where $\Gamma \in \K(\ol \O)$ and an associated displacement $u_\Gamma \in LD(\O \setminus (\Gamma_0 \cup \Gamma))$ solving 
$$\min\left\{\frac12 \int_{\O \setminus (\Gamma_0\cup \Gamma)} \CC e(v):e(v)\, dx : v \in LD(\O\setminus (\Gamma_0\cup \Gamma)), \; v=\psi \text{ on }\partial\O \setminus (\Gamma_0\cup \Gamma)  \right\}.$$
We denote by
\begin{equation}
\G(\Gamma):=\frac12 \int_\O \big[ \CC e(u_\Gamma):e(u_\Gamma)- \CC e(u_0):e(u_0)\big] \, dx \leq 0, \label{defGcal}
\end{equation}
and
\begin{equation}
G_\e :=\frac{1}{\e} \inf\big\{ \G(\Gamma) : \Gamma \in \K(\ol \O),\; \HH^1(\Gamma)Ê\leq \e\big\}. \label{defG}
\end{equation}

\section{Construction of dual functions}\label{sec:prel}

\noindent The goal of this section is to construct the harmonic conjugate and the Airy function associated to the displacement $u_0$ in a neighborhood of the crack tip which is assumed to be the origin. Their construction rests on an abstract functional analysis result (Lemma \ref{lem:XY} below) which puts in duality gradients and functions with vanishing divergence outside an (non-smooth) crack.

\medskip

Let $B=B_{R_0}$ and $B'=B_{R'_0}$ be open balls centered at the origin with radii $R_0<R'_0$, such that $\ol{B'} \subset \O$ and $\partial B' \cap \Gamma_0\neq \emptyset$. By assumption, since $\Gamma_0 \in \K(\ol \O)$ satisfies \eqref{eq:dens1/2}, this property certainly holds true provided $R'_0$ is small enough. Note in particular that the connectedness of $\Gamma_0$ ensures that $\partial B \cap \Gamma_0\neq \emptyset$ as well.

The following result is a generalization of \cite[Lemma 1]{C}.

\begin{lem}\label{lem:XY}
Consider the following subspaces of $L^2(B;\R^2)$:
\begin{eqnarray*}
X & := & \{\sigma \in \C^\infty(\overline B; \R^2) : {\rm supp}(\sigma) \cap \Gamma_0 = \emptyset,\; \div \sigma=0 \text{ in }B\},\\
Y & := & \{\nabla v : v \in H^1(B \setminus \Gamma_0),\; v=0 \text{ on }\partial B \setminus \Gamma_0\}.
\end{eqnarray*}
Then $X^\perp=\overline Y$.
\end{lem}

\begin{proof}
Let $\sigma \in X$ and $v \in H^1(B \setminus \Gamma_0)$ be such that $v=0$ on $\partial B \setminus \Gamma_0$. Consider an open set $U_0 \subset B$ with Lipschitz boundary such that $\Gamma_0 \subset U_0$ and $\ol{U_0} \cap {\rm supp}(\sigma)=\emptyset$. In particular, $B \setminus \ol U_0$ has Lipschitz boundary as well, and thanks to the integration by parts formula in $H^1(B \setminus \ol U_0)$ we infer that
\begin{multline*}
\int_{B \setminus \Gamma_0} \sigma \cdot \nabla v\, dx = \int_{B \setminus \ol U_0} \sigma \cdot \nabla v\, dx \\
=- \int_{B \setminus \ol U_0}(\div \sigma) v\, dx + \int_{\partial B \setminus \ol U_0} (\sigma \nu)  v \, d\HH^1 +  \int_{B \cap \partial U_0} (\sigma \nu)  v \, d\HH^1=0.
\end{multline*}
Indeed, the first integral vanishes since $\div \sigma=0$ in $B$. In addition, both boundary integrals vanish as well since $v=0$ on $\partial B\setminus\Gamma_0$, and $\ol{U_0} \cap {\rm supp}(\sigma)=\emptyset$. Consequently, $X \subset Y^\perp$, and thus $\ol X \subset Y^\perp$.

We next establish the converse inclusion. Let $\Psi \in X^\perp$, then in particular, for any $\sigma_1 \in \C^\infty_c(B \setminus \Gamma_0;\R^2)$ with $\div \sigma_1=0$ in $B \setminus \Gamma_0$ (which implies that $\sigma_1 \in X$),
$$\int_{B \setminus \Gamma_0} \Psi\cdot\sigma_1\, dx =0.$$
According to De Rham's Theorem (see \cite[page 20]{T2}), we get the existence of some $v \in L^2_{\rm loc}(B \setminus \Gamma_0)$ such that $\Psi=\nabla v$ a.e. in $B \setminus \Gamma_0$. Now if $U$ is a smooth open set such that $\ol U \cap \Gamma_0= \emptyset$ and $U \cap \partial B \neq \emptyset$, then the open set $U \cap B$ is Lipschitzian. Thus, for any $\sigma_2 \in \C_c^\infty(U \cap B)$ with $\div \sigma_2=0$ in $U \cap B$ (which implies that $\sigma_2 \in X$ if it is extended by zero on $B \setminus U$),
$$\int_{B \cap U} \Psi\cdot \sigma_2\, dx =0.$$
Applying once more De Rham's Theorem (see \cite[page 19]{T2}), one can find some $v_U \in L^2(B \cap U)$ such that $\Psi=\nabla v_U$ a.e. in $B \cap U$. Therefore $v=v_U+c_U$ a.e. in $B \cap U$ for some constant $c_U \in \R$, and thus $v \in L^2(B \cap U)$. Since $v \in H^1(B \cap U)$, thanks to the integration by parts formula in $H^1(B \cap U)$, we get that for any $\sigma \in \C^\infty_c(U;\R^2)$ with $\div \sigma=0$ in $U$ (which also belongs to $X$ if it is extended by zero on $B \setminus U$),
$$\int_{U \cap \partial B}v\, (\sigma\nu)\, d\HH^1= \int_{\partial (B\cap U)}v \, (\sigma\nu)\, d\HH^1= \int_{B \cap U} \sigma\cdot \nabla v\, dx + \int_{B \cap U}v\,  \div \sigma\, dx =0.$$ 
By density (see {\it e.g.}  \cite[Theorem 1.4]{T2}), we get that for any $\sigma \in L^2(U;\R^2)$ with $\div \sigma=0$ in $H^{-1}(U)$ and $\sigma\nu =0$ in $H^{-1/2}(\partial U)$, then 
$$\langle \sigma\nu,v\rangle_{[H^{1/2}(U \cap \partial B)]',H^{1/2}(U \cap \partial B)}=0.$$
Finally, according to Lemma \ref{lem:approx-norm-trace} below, we deduce that for any $g \in L^2(U \cap \partial B)$,
$$\int_{U \cap \partial B} gv\, d\HH^1=0$$
which shows that $v=0$ on $U \cap \partial B$. Considering now the truncated function $v_k:=(-k \vee v) \wedge k$, where $k \in \N$, we get that $v_k \in H^1(B \setminus \Gamma_0)$, $v_k=0$ on $\partial B \setminus \Gamma_0$, and thus $\nabla v_k \in Y$. Moreover, since  $\nabla v_k \to \nabla v=\Psi$ strongly in $L^2(B;\R^2)$ as $k \to \infty$ we get that $X^\perp \subset \ol Y$ and that $Y^\perp=(\ol Y)^\perp \subset (X^\perp)^\perp =\ol X$.
\end{proof}

\subsection{The harmonic conjugate} 

We are now in position to construct the harmonic conjugate $v_0$ associated to $u_0$ in $B$. By construction, the displacement $u_0$ satisfies a Neumann condition on the crack $\Gamma_0$, while its associated stress $\sigma_0$ has zero divergence outside the crack, both in a weak sense. The harmonic conjugate $v_0$ is, roughly speaking, a dual function of $u_0$ in the sense that it satisfies a homogeneous Dirichlet boundary condition on the crack $\Gamma_0$, and its rotated gradient coincides with the  stress $\sigma_0$. The harmonic conjugate will be of use in the proof of Proposition \ref{PROPbound} in order to prove a lower bound on the energy release rate. It will also appear in the construction of the Airy function.

\begin{prop}\label{prop:harm-conj}
There exists a function $v_0 \in H^1_{0,\Gamma_0}(B;\R^2) \cap  \C^\infty(\ol B \setminus \Gamma_0;\R^2)$ such that 
\begin{equation}\label{eq:harm-conj}
\nabla v_0=\sigma_0^\perp:=\left(\!\!\!\begin{array}{cc}
-(\sigma_0)_{12} &  (\sigma_0)_{11}\\
-(\sigma_0)_{22} & (\sigma_0)_{12}
\end{array}\!\!\!\right) 
\quad \text{ in } B\setminus \Gamma_0.
\end{equation}
\end{prop}

\begin{proof}
According to the variational formulation \eqref{eq:var-form}, for any $v \in H^1(B\setminus \Gamma_0;\R^2)$ with $v=0$ on $\partial B \setminus \Gamma_0$, we have
$$\int_B \sigma_0 : \nabla v\, dx=0.$$
Consequently, both lines of $\sigma_0$, denoted by 
$$\sigma^{(1)}:=\left(\!\!\!\begin{array}{c} (\sigma_0)_{11}\\(\sigma_0)_{12} \end{array}\!\!\!\right), \quad \sigma^{(2)}:=\left(\!\!\!\begin{array}{c}(\sigma_0)_{12}\\(\sigma_0)_{22} \end{array}\!\!\!\right),$$ 
belong to $Y^\perp$. Therefore, Lemma \ref{lem:XY} ensures the existence of a sequence $(\sigma_n^{(1)}) \subset X$ such that $\sigma_n^{(1)} \to \sigma^{(1)}$ in $L^2(B;\R^2)$. Since $\div \sigma_n^{(1)}=0$ in $B$ and ${\rm supp}(\sigma_n^{(1)}) \cap \Gamma_0=\emptyset$, it follows  that 
$$(\sigma_n^{(1)})^\perp:=\left(\!\!\!\begin{array}{c} -(\sigma^{(1)}_n)_{2}\\(\sigma_n^{(1)})_{1} \end{array}\!\!\!\right)=\nabla p_n^{(2)}$$
for some $p_n^{(2)} \in \C^\infty(\ol B)$ with ${\rm supp}(p_n^{(2)}) \cap \Gamma_0=\emptyset$. Consequently, by the Poincar\'e inequality, we get that $p_n^{(2)} \to p^{(2)}$ in $H^1(B)$ for some $p^{(2)} \in H^1_{0,\Gamma_0}(B)$ satisfying $\nabla p^{(2)}=(\sigma^{(1)})^\perp$. We prove similarly the existence of $p^{(1)} \in H^1_{0,\Gamma_0}(B)$ satisfying $\nabla p^{(1)}=-(\sigma^{(2)})^\perp$. We then define 
$$v_0:=\left(\!\!\!\begin{array}{c} p^{(2)} \\-p^{(1)} \end{array}\!\!\!\right)\in H^1_{0,\Gamma_0}(B;\R^2)$$
which satisfies \eqref{eq:harm-conj}. Finally, since $\sigma_0 \in \C^\infty(\ol B \setminus \Gamma_0;\Ms)$, then $v_0\in \C^\infty(\ol B \setminus \Gamma_0;\R^2)$.
%On the other hand, we also have that $\sigma_0 \in Y^\perp=\ol X$ by Lemma \ref{lem:XY}. Consequently, there exists a sequence $(\sigma_n) \subset X$ such that $\sigma_n \to \sigma_0$ in $L^2(B;\Ms)$ as $n \to \infty$. Since, for each $n \in \N$, $\sigma_n \in \C^\infty(\ol B;\Ms)$ and $\div \sigma_n=0$ in $B$, there exists $v_n \in \C^\infty(\ol B;\R^2)$ such that $\sigma_n=\nabla^\perp v_n$. Moreover, since $\sigma_n=0$ in a neighborhood of $\Gamma_0$, one also an assume without loss of generality that $v_n=0$ in a neighborhood of $\Gamma_0$. Finally, $\nabla v_n=\sigma_n^\perp \to \sigma_0^\perp=\nabla v_0$ in $L^2(B;\Ms)$, and according to the Poincar\'e inequality, $v_n \to v_0$ in $L^2(B;\R^2)$ so that $v_0 \in H^1_{\Gamma_0}(B;\R^2)$.
\end{proof}

\subsection{The Airy function}

%According to the variational formulation \eqref{eq:var-form}, one has
%$$\int_{B \setminus \Gamma_0} \sigma_0: \nabla v\, dx =0$$
%for any $v \in H^1(B \setminus \Gamma_0;\R^2)$ with $v=0$ on $\partial B \setminus \Gamma_0$. Consequently, both lines of $\sigma_0$, denoted by 
%$$\sigma^{(1)}:=\left(\!\!\!\begin{array}{c} (\sigma_0)_{11}\\(\sigma_0)_{12} \end{array}\!\!\!\right), \quad \sigma^{(2)}:=\left(\!\!\!\begin{array}{c}(\sigma_0)_{12}\\(\sigma_0)_{22} \end{array}\!\!\!\right)$$ 
%belong to $Y^\perp$. According to Lemma \ref{lem:XY}, there exists a sequence $(\sigma_n^{(1)}) \subset X$ such that $\sigma_n^{(1)} \to \sigma^{(1)}$ in $L^2(B;\R^2)$. Since $\div \sigma_n^{(1)}=0$ in $B$ and $\sigma_n^{(1)}$ vanishes in a neighborhood of $\Gamma_0$, it follows  that 
%$$\nabla^\perp p_n^{(2)}=\sigma_n^{(1)}$$
%for some $p_n^{(2)} \in \C^\infty(\ol B)$ vanishing in a neighborhood of $\Gamma_0$. Consequently, $p_n^{(2)} \to p^{(2)}$ in $H^1(B)$ for some $p^{(2)} \in H^1_{\Gamma_0}(B)$ satisfying $\nabla^\perp p^{(2)}=\sigma^{(1)}$. We prove similarly the existence of $p^{(1)} \in H^1_{\Gamma_0}(B)$ satisfying $\nabla^\perp p^{(1)}=-\sigma^{(2)}$.

We next construct the Airy function $w_0$ associated to the displacement $u_0$ in $B$ following an approach similar to \cite{C}. This new function has the property to be a biharmonic function vanishing on the crack. Therefore, the original elasticity problem \eqref{minprob1} can be recast into a suitable biharmonic equation whose associated natural energy (the $L^2$ norm of the hessian) coincides with the original elastic energy. The Airy function will be useful in section \ref{sec:blow-up} in order to get an {\it a priori} bound on the rescaled elastic energy around the crack tip, as well as in our convergence result for the blow-up displacement.

\begin{prop}\label{prop:airy}
There exists a function $w_0 \in H^2_{0,\Gamma_0}(B)$ such that 
\begin{equation}\label{eq:biharm}
\Delta^2w_0 =0 \text{ in }\D'(B \setminus \Gamma_0)
\end{equation}
and
\begin{equation}\label{eq:hessian}
D^2 w_0 = 
\left(\!\!\!\begin{array}{cc}
(\sigma_0)_{22} & - (\sigma_0)_{12}\\
- (\sigma_0)_{12} & (\sigma_0)_{11}
\end{array}\!\!\!\right).
\end{equation}
\end{prop}

\begin{proof}
We reproduce the construction initiated in the proof of Proposition \ref{prop:harm-conj} with the larger ball $B'$ instead of $B$. It ensures the existence of $p^{(1)}$ and $p^{(2)} \in  H^1_{0,\Gamma_0}(B')$ such that
$$\nabla p^{(1)}=\left(\!\!\!\begin{array}{c} (\sigma_0)_{22}\\-(\sigma_0)_{12} \end{array}\!\!\!\right), \qquad \nabla p^{(2)}=\left(\!\!\!\begin{array}{c} -(\sigma_0)_{12}\\(\sigma_0)_{11} \end{array}\!\!\!\right).$$

By definition, there exists sequences $(p_n^{(1)})$ and $(p_n^{(2)}) \subset \C^\infty(\ol{B'})$ vanishing in a neighborhood of $\Gamma_0$ in $\ol{B'}$, and such that $p_n^{(1)} \to p^{(1)}$ and $p_n^{(2)} \to p^{(2)}$ in $H^1(B')$. For any $v \in H^1(B' \setminus \Gamma_0)$ with $v=0$ on $\partial B' \setminus \Gamma_0$, we infer thanks to the integration by parts formula that
\begin{multline*}
\int_{B'} \left(\!\!\!\begin{array}{c}-p^{(2)}\\p^{(1)} \end{array}\!\!\!\right)\cdot \nabla v\, dx =\int_{B'}(-p^{(2)} \partial_1 v +p^{(1)} \partial _2 v)\, dx
=\lim_{n \to \infty}\int_{B'}(-p_n^{(2)} \partial_1 v +p_n^{(1)} \partial _2 v)\, dx\\
=\lim_{n \to \infty}\int_{B'}(-\partial_1 p_n^{(2)} + \partial _2 p_n^{(1)} )v\, dx=\int_{B'}(-\partial_1 p^{(2)} + \partial _2 p^{(1)} )v\, dx=0.
\end{multline*}
Therefore, it follows that 
$$\left(\!\!\!\begin{array}{c}-p^{(2)}\\p^{(1)} \end{array}\!\!\!\right) \in Y^\perp=\ol X$$
according again to Lemma \ref{lem:XY}. Arguing as in the proof of Proposition \ref{prop:harm-conj}, we deduce the existence of some $w_0 \in H^1_{0,\Gamma_0}(B')$ such that 
$$\nabla w_0=\left(\!\!\!\begin{array}{c}p^{(1)}\\p^{(2)} \end{array}\!\!\!\right).$$
By construction, the Airy function $w_0$ satisfies \eqref{eq:hessian}. Consequently, $w_0 \in H^1_{0,\Gamma_0}(B') \cap H^2(B')$ with $\nabla w_0 \in H^1_{0,\Gamma_0}(B';\R^2)$. 

\vskip5pt

Let us show that $w_0 \in H^2_{0,\Gamma_0}(B)$. This property rests on a capacity argument similar to that used in \cite[Theorem 1]{C}. We first observe that since $w_0 \in H^2(B')$, it is therefore (H\"older) continuous by the Sobolev imbedding, so that it makes sense to consider its pointwise values. Let us consider a cut-off function $\eta \in \C_c^\infty(B';[0,1])$ satisfying $\eta=1$ on $\ol B$. Denoting $z_0:=\eta w_0$, then $z_0 \in H^1_0(B' \setminus \Gamma_0)$ and $\nabla z_0 \in H^1_0(B'\setminus \Gamma_0;\R^2)$. 
As a consequence of \cite[Theorem 3.3.42]{HP}, the function $\nabla z_0$ has a $\Cap_{1,2}$-quasicontinuous representative, denoted by $\widetilde{\nabla z_0}$, satisfying  $\widetilde{\nabla z_0} = 0$ $\Cap_{1,2}$-q.e. on $\partial (B' \setminus \Gamma_0)$. %(actually $\nabla z_0=0$ in a neighborhood of $\partial B'$). 
We next show that the function $z_0$ has a $\Cap_{2,2}$-quasicontinuous representative vanishing $\Cap_{2,2}$-q.e. on $\partial (B' \setminus \Gamma_0)$. Note that since the empty set is the only set of zero $\Cap_{2,2}$-capacity, it is equivalent to show that $z_0 = 0$ everywhere on $\partial (B' \setminus \Gamma_0)$. As before, since $z_0 \in H^1_0(B' \setminus \Gamma_0)$, we deduce that $z_0$ has a $\Cap_{1,2}$-quasicontinuous representative, denoted by $\widetilde{z_0}$, satisfying  $\widetilde{z_0} = 0$ $\Cap_{1,2}$-q.e. on $\partial (B' \setminus \Gamma_0)$. Therefore, defining $K:=\{x \in \partial (B' \setminus \Gamma_0) : z_0(x)=0\}$, then $K$ is a compact set satisfying $\Cap_{1,2}(\partial (B' \setminus \Gamma_0)Ê\setminus K)=0$. Let $\gamma$ be a connected component of $\partial (B' \setminus \Gamma_0)Ê\setminus K$. Since a compact and connected set of positive diameter has a positive $\Cap_{1,2}$-capacity (see \cite[Corollary 3.3.25]{HP}, we deduce that $\diam(\gamma)=\diam(\bar \gamma)=0$ so that $\gamma$ is (at most) a singleton. Moreover, $K$ being compact, its complementary $ \partial (B' \setminus \Gamma_0)Ê\setminus K$ is open in the relative topology of $ \partial (B' \setminus \Gamma_0)$, and thus $\gamma$ is (at most) an isolated point. Finally since $ \partial (B' \setminus \Gamma_0)$ is connected, it turns out that $\gamma=\emptyset$ and thus $z_0=0$ on $\partial (B' \setminus \Gamma_0)$. As a consequence of \cite[Theorem 9.1.3]{AH}, we get that $z_0 \in H^2_0(B' \setminus \Gamma_0)$, or in other words, that there exists a sequence $(z_n) \subset \C^\infty_c(B' \setminus \Gamma_0)$ such that $z_n \to \eta z_0$ in $H^2(B' \setminus \Gamma_0)$. Note in particular that $z_n \in \C^\infty(\ol B)$ and that $z_n$ vanishes in a neighborhood of $\Gamma_0$ in $\ol B$. Therefore, since $z_0 =w_0$ and $\nabla z_0=\nabla w_0$ in $B$, we deduce that $w_0 \in H^2_{0,\Gamma_0}(B)$.

\vskip5pt

We next show that $w_0$ is a biharmonic function. Indeed, according to  \eqref{eq:hessian}, one has
$$\Delta^2w_0 =\Delta ((\sigma_0)_{11}+(\sigma_0)_{22}) \quad \text{ in }Ê\D'(B \setminus \Gamma_0).$$
Denoting by $e_0:=e(u_0)$ the elastic strain, and using the compatibility condition 
$$2\partial^2_{12} (e_0)_{12} =\partial^2_{11} (e_0)_{22} +\partial^2_{22} (e_0)_{11} \quad \text{ in }Ê\D'(B \setminus \Gamma_0)$$
together with Hooke's law \eqref{eq:Hooke-1},
\begin{eqnarray*}
(e_0)_{11}& = & \frac{(\sigma_0)_{11}}{E} - \frac{\nu}{E}(\sigma_0)_{22},\\
(e_0)_{22}& = & \frac{(\sigma_0)_{22}}{E} - \frac{\nu}{E}(\sigma_0)_{11},\\
(e_0)_{12}& = & \frac{1+\nu}{E}(\sigma_0)_{12},
\end{eqnarray*}
we infer that 
$$\Delta^2w_0 =(1+\nu) [ \partial_{11}^2 (\sigma_0)_{11}+ \partial_{22}^2 (\sigma_0)_{22}+2 \partial_{12}^2 (\sigma_0)_{12}] \quad \text{ in }Ê\D'(B \setminus \Gamma_0).$$
Finally, according to the variational formulation \eqref{eq:var-form}, we have
$$\div \sigma_0=0  \quad \text{ in }Ê\D'(B \setminus \Gamma_0)$$
from which \eqref{eq:biharm} follows.
\end{proof}

\begin{rem}\label{rem:airy-min}
The biharmonicity \eqref{eq:biharm} of the Airy function $w_0$ is equivalent to the following local minimality property
$$\int_B |D^2 w_0|^2\, dx \leq \int_B |D^2 z|^2\, dx,$$
for all $z \in w_0 + H^2_0(B)$.
\end{rem}

\begin{rem}\label{rem:nrj-w_0}
According to the results of \cite{KKLO}, we get the following estimate of the energy of $w_0$ around the origin: for every $2\varrho < R \leq R_0$,
$$\int_{B_\varrho} |D^2 w_0|^2\, dx \leq \frac{C_0 \varrho}{R}\int_{B_R} |D^2 w_0|^2\, dx,$$
for some universal constant $C_0>0$ independent of $R$ and $\varrho$. Indeed, it suffices to apply  \cite[Theorem 2]{KKLO} in the open set $B \setminus \Gamma_0$ with (in their notation) $\omega=2\pi$ and $\delta=1/2$. This is possible since, $\Gamma_0$ being connected, then for all $\varrho<R$ we have $\partial B_\varrho \cap \Gamma_0\neq \emptyset$, $\HH^1(\partial B_\varrho \setminus \Gamma_0) \leq 2\pi \varrho$ and $\partial (B \setminus \Gamma_0) \cap \partial (B_R \setminus \Gamma_0) = \Gamma_0 \cap B_R \subset \Gamma_0 \cap B$. 
\end{rem}

Thanks to the reformulation of the elasticity problem as a biharmonic equation, and according to Remark \ref{rem:nrj-w_0} concerning the behavior of the energy of a biharmonic function in fractured domains, we get the following result about the elastic energy concentration around the crack tip. We observe that in \cite{CL} a stronger result has been obtained in the scalar (anti-plane) case where a monotonicity formula has been established.

\begin{prop}\label{prop:est-sigma} Let $\sigma_0$ be the stress defined in \eqref{defsigma0} and $R_0>0$ be such that $B_{R_0}\subset \Omega$ and  $\partial B_{R_0}\cap \Gamma \not = \emptyset$. Then there exists a universal constant $C_0>0$ such that for all $\rho$, $R>0$ satisfying $2\varrho < R \leq R_0$,
$$\int_{B_\varrho} |\sigma_0|^2\, dx \leq \frac{C_0 \varrho}{R}\int_{B_R} |\sigma_0|^2\, dx.$$
\end{prop}

\begin{proof}
%Proposition \ref{prop:airy} ensures that $w \in H^2_{0,\Gamma_0}(B)$ solves the variational formulation
%$$\int_{B \setminus \Gamma_0} D^2w : D^2\varphi\, dx=0 \quad \text{ for all }\varphi \in H^2_0(B \setminus \Gamma_0).$$
%We next want to apply the results of \cite{KKLO} in the open set $B \setminus \Gamma_0$. In the notations of \cite{KKLO}, we have $\omega=2\pi$ and $\delta=1/2$. Moreover, since $\Gamma_0$ is connected, then for all $\varrho<R$ we have $\partial B_\varrho \cap \Gamma_0\neq \emptyset$, $\HH^1(\partial B_\varrho \setminus \Gamma_0) \leq 2\pi \varrho$ and $\partial (B \setminus \Gamma_0) \cap \partial (B_R \setminus \Gamma_0) = \Gamma_0 \cap B_R \subset \Gamma_0 \cap B$. Therefore, according to \cite[Theorem 2]{KKLO}, we deduce that
%$$\int_{B_\varrho} |D^2 w|^2\, dx \leq \frac{C_0 \varrho}{R}\int_{B_R} |D^2 w|^2\, dx,$$
%for some universal constant $C_0>0$. 
The result is an immediate consequence of \eqref{eq:hessian} together with Remark \ref{rem:nrj-w_0}.
\end{proof}

\section{Bounds on the energy release rate}\label{sec:bounds}

The goal of this section is to establish bounds on the energy release rate. This is the first step toward a more precise analysis and a characterization of the energy release rate as a limiting minimization problem (see section \ref{sec:err}). As in \cite[Lemma 2.4]{CFM}, the proof of the upper bound relies on the construction of an explicit competitor for the minimization problem \eqref{defG} defining $G_\e$. The lower bound rests in turn into a dual formulation (in term in the stress) of the minimization problem \eqref{defGcal}, and into the construction, for each crack increment, of an admissible stress competitor for this new dual variational problem. The construction we use is based on the harmonic conjugate $v_0$ associated to the displacement obtained in Proposition~\ref{prop:harm-conj}.

\begin{prop}\label{PROPbound}
There exist two constants $0<G_*\leq G^* <\infty$ such that
$$-G^*  \leq \liminf_{\eÊ\to 0} G_\e\leq \limsup_{\eÊ\to 0} G_\e \leq - G_*.$$
\end{prop}

\begin{proof}
{\bf Upper bound.} Since $0 \in \O$, one can choose $\e>0$ small enough so that $\ol{B}_{\e/(2\pi+1)} \subset \O$. Let 
$$\Gamma:= \partial B_{\e/(2\pi+1)} \cup \{ (t,0) : 0 \leq t \leq \e/(2\pi+1)\}.$$
This set clearly belongs to $\K(\ol \O)$ and $\HH^1(\Gamma) =Ê\e$. Defining $v:=u_0 \chi_{\O \setminus \ol{B}_{\e/(2\pi+1)}}$, we infer that $v \in LD(\O \setminus (\Gamma_0 \cup \Gamma))$ with $v=u_0=\psi$ on $\partial \O \setminus (\Gamma_0 \cup \Gamma)$. Consequently,
$$\G(\Gamma) \leq \frac{1}{2} \int_\O\big[ \CC e(v):e(v)\, dx -\CC e(u_0):e(u_0)\big] \, dx = - \frac{1}{2} \int_{B_{\e/(2\pi+1)}}\CC e(u_0):e(u_0)\, dx.$$
We then apply Proposition \ref{prop:est-sigma} which shows that
$$\limsup_{\e \to 0} G_\e \leq -G_*,$$
for some $G_*>0$.

\vskip10pt

\noindent {\bf Lower bound.} Let $\e>0$ be small enough so that $2\eÊ\leq R_0$, $\ol{B}_{2\e} \subset \O$ and $2\e \not\in \mathcal N$, where $\mathcal N$ is the exceptional set given by Lemma \ref{lem:IPP} below. According to \cite[p. 330]{CGP}, for any $\Gamma \in \K(\ol \O)$ with $\HH^1(\Gamma) \leq \e$, one has 
\begin{equation}\label{eq:bd-below}
\frac12 \int_\O\big[ \CC e(u_\Gamma):e(u_\Gamma)-\CC e(u_0):e(u_0)\big] \, dx
\geq - \frac12 \int_\O (\tau - \sigma_0): \CC^{-1}(\tau-\sigma_0)\, dx
\end{equation}
for every statically admissible stresses $\tau \in L^2(\O;\M)$ satisfying 
\begin{equation}\label{eq:stat-admiss}
\int_\O \tau:e(v)\, dx =0\quad \text{ for any } v \in LD(\O \setminus (\Gamma_0 \cup \Gamma)) \text{ with }v=0 \text{ on }Ê\partial \O \setminus (\Gamma_0 \cup \Gamma).
\end{equation}
We now construct a convenient competitor $\tau$ for \eqref{eq:stat-admiss}. Since $\Gamma$ is connected, $0 \in \Gamma$ and $\HH^1(\Gamma) \leq \e$ it follows that $\GammaÊ\subset B_\e$. Let $\eta \in \C^\infty_c(\O;[0,1])$ be a cut-off function satisfying 
$$
\begin{cases}
\eta=1 & \text{ in }ÊB_{5\e/4},\\
\eta=0 & \text{ in }Ê\O \setminus B_{7\e/4},\\
\|\nabla \eta\|_\infty \leq 3/\e.
\end{cases}
$$
We define $\tau \in L^2(\O;\M)$ by
\begin{equation}\label{eq:def-tau}
\tau=
\begin{cases}
0 & \text{ in }ÊB_\e,\\
\nabla^\perp((1-\eta)v_0) & \text{ in }ÊB_{2\e} \setminus B_\e,\\
\sigma_0 & \text{ in }\O \setminus B_{2\e},
\end{cases}
\end{equation}
where $v_0$ is the harmonic conjugate of $u_0$ in the ball $B=B_{R_0}$. Let us check that $\tau$ satisfies \eqref{eq:stat-admiss}. By the density result \cite[Theorem 1]{C}, it is enough to consider test functions $v \in H^1(\O \setminus (\Gamma_0 \cup \Gamma);\R^2)$ with $v=0$ on $\partial \O \setminus (\Gamma_0 \cup \Gamma)$. Then 
\begin{equation}\label{eq:tau1}
\int_\O \tau:e(v)\, dx= \int_{B_{2\e} \setminus B_\e} \nabla^\perp((1-\eta)v_0) : e(v)\, dx + \int_{\O \setminus B_{2\e} }\sigma_0:e(v)\, dx.
\end{equation}
Since $\Gamma \subset B_\e$, then actually $v=0$ on $\partial \O \setminus \Gamma_0$, and the second integral writes
\begin{equation}\label{eq:tau2}
\int_{\O \setminus B_{2\e} }\sigma_0:e(v)\, dx=-\int_{\partial B_{2\e}Ê\setminus \Gamma_0} (\sigma_0 \nu)\cdot v\, d\HH^1
%-\langle \sigma_0 \nu, v \rangle_{[H^{1/2}(\partial B_{2\rho} \setminus \Gamma_0)]',H^{1/2}(\partial B_{2\rho} \setminus \Gamma_0)},
\end{equation}
thanks to the integration by parts formula given by Lemma \ref{lem:IPP} below. To treat the first integral, we recall that there exists a sequence $(v_n)Ê\subset \C^\infty(\ol{B_{2\e}};\R^2)$ with $v_n=0$ in a neighborhood of $\Gamma_0$ and such that $v_n \to v_0$ in $H^1(B_{2\e};\R^2)$. Hence, using an integration by parts, we infer that
\begin{multline*}
\int_{B_{2\e} \setminus B_\e} \nabla^\perp((1-\eta)v_0) : e(v)\, dx=\lim_{n \to \infty}\int_{B_{2\e} \setminus B_\e} \nabla^\perp((1-\eta)v_n) : e(v)\, dx\\
=\lim_{n \to \infty}\Bigg[ -\int_{B_{2\e} \setminus B_\e} \big(\div\nabla^\perp((1-\eta)v_n)\big) \cdot v\, dx\\
+ \int_{\partial B_{2\e}} \big(\nabla^\perp ((1-\eta)v_n)\nu\big) \cdot v \, d\HH^1 - \int_{\partial B_\e} \big(\nabla^\perp ((1-\eta)v_n)\nu\big) \cdot v \, d\HH^1 \Bigg].
\end{multline*}
But since $\div(\nabla^\perp((1-\eta)v_n))=0$ in $B_{2\e}$, $\eta=1$ in a neighborhood of $\partial B_\e$ and $\eta=0$ in a neighborhood of $\partial B_{2\e}$, we deduce that
$$\int_{B_{2\e} \setminus B_\e} \nabla^\perp((1-\eta)v_0) : e(v)\, dx=\lim_{n \to \infty}\int_{\partial B_{2\e}} (\nabla^\perp v_n \nu) \cdot v \, d\HH^1.$$
Denoting $\sigma_n=\nabla^\perp v_n\in \C^\infty(\ol{B}_{2\e};\M)$, then $\div \sigma_n=0$ in $B_{2\e}$ and $\sigma_n \to  \nabla^\perp v_0=\sigma_0$ in $L^2(B_{2\e};\M)$ so that   $\sigma_n \nu \to\sigma_0\nu$ in $H^{-1/2}(\partial B_{2\e};\R^2)$. We therefore deduce that
\begin{equation}\label{eq:tau3}
\int_{B_{2\e} \setminus B_\e} \nabla^\perp((1-\eta)v_0) : e(v)\, dx=\int_{B_{2\e}Ê\setminus \Gamma_0} (\sigma_0 \nu)\cdot v\, d\HH^1.
%\langle \sigma_0 \nu, v \rangle_{[H^{1/2}(\partial B_{2\rho} \setminus \Gamma_0)]',H^{1/2}(\partial B_{2\rho} \setminus \Gamma_0)}.
\end{equation}
Gathering \eqref{eq:tau1}, \eqref{eq:tau2} and \eqref{eq:tau3}, we finally conclude that the admissibility condition \eqref{eq:stat-admiss} holds.

Taking $\tau$ defined by \eqref{eq:def-tau} as competitor in \eqref{eq:bd-below} and recalling that $\sigma_0=\nabla^\perp v_0$, we infer that
\begin{equation}\label{eq:lw-bd}
\frac12 \int_\O\big( \CC e(u_\Gamma):e(u_\Gamma)-\CC e(u_0):e(u_0)\big) \, dx
\geq -c \left( \int_{B_{2\e}} |\sigma_0|^2\, dx + \frac{1}{\e^2} \int_{B_{2\e}Ê\setminus B_\e}|v_0|^2\, dx \right),
\end{equation}
for some constant $c>0$ only depending on the Lam\'e constants $\lambda$ and $\mu$. Let $(v_n)\subset \C^\infty(\ol B;\R^2)$ be such that $v_n \to v_0$ in $H^1(B;\R^2)$ and $v_n=0$ in a neighborhood of $\Gamma_0$. For each $n \in \N$, the coarea formula says that 
$$\int_{B_{2\e}Ê\setminus B_\e}|v_n|^2\, dx=\int_\e^{2\e} \int_{\partial B_r} |v_n|^2\, d\HH^1\, dr.$$
But since $v_n=0$ on $\Gamma_0$ and $\Gamma_0$ is connected, for each $r \in [\e,2\e]$, there exists $\xi_r \in \partial B_r \cap \Gamma_0$ (also depending on $n$). Hence, for all $\xi \in \partial B_r$, 
$$v_n(\xi) = \int_{\wideparen{(\xi_r,\xi)}}\partial_\tau v_n\, d\HH^1,$$
where $\wideparen{(\xi_r,\xi)}$ stands for the smaller arc of circle in $\partial B_r$ joining $\xi_r$ and $\xi$, and $\partial_\tau v_n$ is the tangential derivative of $v_n$ on $\partial B_r$. Thus, according to the Cauchy-Schwarz inequality, for all $\xi \in \partial B_r$,
$$|v_n(\xi)|^2 \leq \pi r \int_{\partial B_r}|\partial_\tau v_n|^2\, d\HH^1,$$
and it results, after integration with respect to $\xi$ and $r$ that
$$\int_{B_{2\e}Ê\setminus B_\e}|v_n|^2\, dx \leq 2\pi^2 \int_\e^{2\e} r^2\int_{\partial B_r}|\partial_\tau v_n|^2\, d\HH^1\, dr \leq 8\pi^2 \e^2 \int_{B_{2\e}Ê\setminus B_\e}|\nabla v_n|^2\, dx.$$
Passing to the limit as $n \to \infty$ yields
$$\int_{B_{2\e}Ê\setminus B_\e}|v_0|^2\, dx \leq 8\pi^2 \e^2 \int_{B_{2\e}Ê\setminus B_\e}|\nabla v_0|^2\, dx,$$
and remembering that $|\nabla v_0|=|\nabla^\perp v_0|=|\sigma_0|$,  we finally obtain
$$\int_{B_{2\e}Ê\setminus B_\e}|v_0|^2\, dx \leq  8\pi^2 \e^2 \int_{B_{2\e}Ê\setminus B_\e}|\sigma_0|^2\, dx.$$
Inserting this result into \eqref{eq:lw-bd}, it follows that
$$\frac12 \int_\O \big[\CC e(u_\Gamma):e(u_\Gamma)-\CC e(u_0):e(u_0)\big] \, dx
\geq - c\int_{B_{2\e}} |\sigma_0|^2\, dx$$
for some constant $c>0$ only depending on $\lambda$ and $\mu$. Minimizing the left hand side of the previous inequality with respect to all $\Gamma \in \K(\ol \O)$ with $\HH^1(\Gamma) \leq \e$ yields
$$G_\e \geq - \frac{c}{\e} \int_{B_{2\e}} |\sigma_0|^2\, dx.$$
Then Proposition \ref{prop:est-sigma} shows that
$$\liminf_{\e \to 0} G_\e  \geq -G^*$$
for some $G^*>0$.
\end{proof}

%%%%%%%%%%%%%%%%%%%%%%%%%%%%%%%%%%%%%%

\section{Blow-up limit of the pre-existing crack}\label{sec:blow-up}

\noindent In this section we investigate the nature of the singularity of the displacement $u_0$ and the stress $\sigma_0$ at the origin, which is the tip of the crack $\Gamma_0$ having density $1/2$ at that point. We will prove, that along suitable subsequences of radius $\e_k \to 0$ of balls, the rescaled crack converges in the Hausdorff sense to a half-line (modulo a rotation), and the rescaled displacement converges in a certain sense to the usual crack-tip function in the complement of a half-line. Once again, the analysis strongly relies on the Airy function introduced in Proposition \ref{prop:airy}. Contrary to \cite{CL} where the scalar anti-plane was treated, we do not have any monotonicity formula on the energy (neither for the elastic problem nor for the biharmonic one) which prevents one to ensure the existence of the limit of the rescaled energy, and thus the uniqueness of the limit. Therefore, in contrast with \cite{CL}, our result strongly depends upon the sequence $(\e_n)$.

\medskip

Let $R_0>0$ be such that $B_{R_0}Ê\subset \O$, and $0<\e \leq R_0/2$. According to Proposition 1 and Remark 2 in \cite{CL}, there exists a sequence of rotations $\mathcal R_\e$ such that the rescaled crack
\begin{equation}\label{eq:crack-rescaled}
\Sigma_\e:=\e^{-1} \mathcal R_\e (\Gamma_0 \cap B_\e)
\end{equation}
locally converges to the half line $\Sigma_0:=(-\infty,0]\times \{0\}$ with respect to the Hausdorff distance. 

In this section we are interested in the asymptotic behavior of the rescaled displacement $u_\e \in LD(B_{R_0/\e})$ defined  by
\begin{equation}\label{eq:displacement-rescaled}
u_\e(y):=\e^{-1/2} u_0(\mathcal R_\e^{-1}(\e y)) \quad \text{for every } y \in B_{R_0/\e}.
\end{equation}
To this aim, it will again be convenient to work on the Airy function. Let us consider the Airy function $w_0 \in H^2_{0,\Gamma_0}(B_{R_0})$ associated to $u_0$ in $B_{R_0}$ given by Proposition \ref{prop:airy} satisfying \eqref{eq:biharm} and \eqref{eq:hessian}. The rescaled  Airy function $w_\e\in H^2_{0,\Sigma_\e}(B_{R_0/\e})$ is defined by
\begin{equation}\label{eq:airy-rescaled}
w_\e(y):=\e^{-3/2}w_0(\mathcal R_\e^{-1}(\e y))\quad \text{for every } y \in B_{R_0/\e}.
\end{equation}

\subsection{Blow-up analysis of the Airy function}

We first show that the Airy function blows-up into a biharmonic function outside the half line limit crack, satisfying a homogeneous Dirichlet condition on the crack, and that its energy computed on a ball behaves like the radius.

\begin{prop}\label{thm:blowup-airy}
For every sequence $(\e_n) \searrow 0^+$, there exist a subsequence $(\e_k)Ê\equiv (\e_{n_k})Ê\searrow 0^+$ and $w_{\Sigma_0} \in H^2_{\rm loc}(\R^2)$ such that 
$$w_{\e_k}Ê\to  w_{\Sigma_0}  \text{ strongly in }H^2_{\rm loc}(\R^2).$$
In addition, $w_{\Sigma_0}$ is a solution of the following biharmonic problem with homogeneous Dirichlet boundary condition on the crack:
\begin{eqnarray}
\begin{cases}
\Delta^2 w_{\Sigma_0}=0 \text{ in }\D'(\R^2 \setminus \Sigma_0),\\
w_{\Sigma_0}\in H^2_{0,\Sigma_0}(B_R) \text{ for any }R>0,
\end{cases}\label{prob1}
\end{eqnarray}
and it satisfies the following energy bound
\begin{equation}\label{eq:bound-u_sigma_0}
\sup_{R>0}\frac{1}{R}\int_{B_R} |D^2 w_{\Sigma_0}|^2\, dx <\infty.
\end{equation}
%Moreover, the function $w_{\Sigma_0}$ satisfies
%$$w_{\Sigma_0}=\alpha \varphi + \beta \psi,$$
%where $\alpha$ and $\beta \in \R$ are constants, while $\varphi$ and $\psi$ are the homogeneous functions defined in polar coordinates by
%\begin{eqnarray*}
%\varphi(r,\theta) & :=& r^{3/2}\left[\sin\left(Ê\frac{3\theta}{2}\right) -3 \sin \left(\frac{\theta}{2} \right)\right],\\
%\psi(r,\theta) & :=& r^{3/2} \left[\cos\left(Ê\frac{3\theta}{2}\right) - \cos \left(\frac{\theta}{2} \right)\right]
%\end{eqnarray*}
%for every $r >0$ and every $\theta \in [-\pi,\pi]$.
\end{prop}

\begin{proof} 
The proof is divided into several steps. We first derive weak compactness on the rescaled Airy function, according the energy bound of the original Airy function. We then derive a Dirichlet condition on the crack for the weak limit and its gradient. Using a cut-off function argument, we establish that the weak convergence is actually strong, which enables one to show that the limit Airy function is a biharmonic function outside the crack. In the sequel $R>0$ is fixed, and $\e>0$ is small enough such that $2R<R_0/\e$. 
%Then using results of \cite{Gr1}, we get an expression of the limit as the sum of a homogeneous function and a higher regular remaining term. Eventually, we perform an additional blow-up in order to get rid of the remaining term.
%\medskip

\noindent {\bf Weak compactness.} According to \cite[Theorem 2]{KKLO}, we have
\begin{equation}\label{eq:airy-bound}
\int_{B_{2R}} |D^2 w_\e(y)|^2 \, dy = \e \int_{B_{2R}}|D^2 w_0(\mathcal R_\e^{-1}(\e y))|^2\, dy = \frac{1}{\e}\int_{B_{2R\e}}|D^2w_0(x)|^2 \, dx \leq C_0 R,
\end{equation}
where $C_0>0$ is independent of $\e$ and $R$. Since $w_\e \in H^2_{0,\Sigma_\e}(B_{2R})$, Poincar\'e inequality 
%\footnote{C'est clair non? On fait des int\'egrales sur des sph\`eres qui doivent forc\'ement couper $\Gamma_\rho$ par connexit\'e}
implies that the sequence $(w_\e)_{\e>0}$ is uniformly bounded in $H^2(B_{2R})$. A standard diagonalisation argument shows that for each sequence $(\e_n) \searrow 0^+$, it is possible to extract a subsequence $(\e_k)\equiv (\e_{n_k}) \searrow 0^+$ and find $w_{\Sigma_0} \in H^2_{\rm loc}(\R^2)$ such that $w_{\e_k}Ê\wto  w_{\Sigma_0}$ weakly in $H^2_{\rm loc}(\R^2)$.  In particular, passing to the $\liminf$ in \eqref{eq:airy-bound} yields \eqref{eq:bound-u_sigma_0}.
%According to \eqref{eq:airy-bound}, we have that
%\begin{equation}\label{eq:airy-bound2}
%\int_{B_{2}} |D^2 \hat  w_{\Sigma_0}|^2 \, dy \leq C_0.
%\end{equation}
In addition, we can assume that, for the same subsequence, $w_{\e_k}Ê\to  w_{\Sigma_0}$ strongly in $H_{\rm loc}^1(\R^2) \cap L^\infty_{\rm loc}(\R^2)$, and that $|D^2w_{\e_k}|^2 \LL^2 \wto \mu$ weakly* in $\MM_{\rm loc}(\R^2)$ for some nonnegative measure $\mu \in \MM_{\rm loc}(\R^2)$.

\medskip

\noindent {\bf Condition on the crack.} %\footnote{Autre possibilit\'e: faire une preuve comme sugg\`ere Antonin en se pla\c cant sur une boule centr\'e en un point situŽ sur l'axe des abscisses (\'eloign\'e de l'origine). On prolonge la fonction $\nabla w_{\rho_k}$ par $0$ de l'autre c\^ot\'e de la boule ce qui donne une fonction $H^2$ de toute la boule qui converge vers une fonction $H^2$ de la boule et nulle sur la demi-boule inf\'erieure. Mais c'est plus long...}
Let us show that $w_{\Sigma_0} \in H^2_{0,\Sigma_0}(B_r)$ for any $r<2R$. Consider a cut-off function $\eta \in \C^\infty_c(B_{2R};[0,1])$ such that $\eta=1$ on $B_r$, and let $z:=\eta w_{\Sigma_0}\in H^2_0(B_{2R})$. Note that since $w_{\e_k} \to w_{\Sigma_0}$ uniformly on $\overline{B}_{2R}$ and $\Sigma_{\e_k} \to \Sigma_0$ in the sense of Hausdorff in $B_{2R}$, then $w_{\Sigma_0}=0$ on $\Sigma_0$, and thus $z=0$ on $\partial (B_{2R} \setminus \Sigma_0)$. On the other hand, since $\nabla (\eta w_{\e_k}) \in H^1_0(B_{2R} \setminus \Sigma_{\e_k};\R^2)$ and $\nabla (\eta w_{\e_k}) \wto \nabla z$ weakly in $H^1(B_{2R};\R^2)$, it follows from \cite{S} that $\nabla z \in H^1_0(B_{2R} \setminus \Sigma_0;\R^2)$. Therefore, $\nabla z$ has a $\Cap_{1,2}$-quasicontinuous representative, denoted by $\widetilde{\nabla z}$, such that $\widetilde{\nabla z}=0$ $\Cap_{1,2}$-q.e. on $\partial (B_{2R} \setminus \Sigma_0)$. As a consequence of \cite[Theorem 9.1.3]{AH} (see also \cite[Theorem 3.8.3]{HP}), we get that $z \in H^2_0(B_{2R} \setminus \Sigma_0)$, and thus that $w_{\Sigma_0} \in H^2_{0,\Sigma_0}(B_r)$.

\medskip

\noindent {\bf Strong convergence.} Our aim now is to prove that $w_{\e_k} \to w_{\Sigma_0}$ strongly in $H_{\rm loc}^2(\R^2)$. %To show this property, it is enough to check that 
%$$\|D^2 w_{\e_k}\|_{L^2(B_R)}\to \|D^2 w_{\Sigma_0}\|_{L^2(B_R)}.$$
By the lower semicontinuity of the norm with respect to weak convergence, we already have for any $r<2R$
\begin{equation}\label{eq:liminf}
\int_{B_r} |D^2  w_{\Sigma_0}|^2\, dx \leq \liminf_{k \to \infty} \int_{B_r} |D^2w_{\e_k}|^2\, dx ,
\end{equation}
so that it is enough to prove the converse inequality with a $\limsup$. To this aim we will use the minimality property of $w_{\e_k}$, and suitably modify $w_{\Sigma_0}$ into an admissible competitor.

Let us select a radius $r \in (R,2R)$ such that $\mu (\partial B_r)=0$. Since $w_{\Sigma_0} \in H^2_{0,\Sigma_0}(B_r)$, for every $n \in \N$, there exists a function $h_n \in \C^\infty(\overline{B_r})$ such that ${\rm supp}(h_n)\cap \Sigma_0=\emptyset$ and $h_n \to w_{\Sigma_0}$ in $H^2(B_r)$ as $n \to \infty$. 
Note that, by Hausdorff convergence, one also has that ${\rm supp}(h_n)\cap \Sigma_{\e_k}=\emptyset$ for $k \geq k_n$ large enough, for some integer $k_n \in \N$.

Let us consider  a cut-off function $\eta_\delta \in \C^\infty_c(B_r ; [0 ,1])$ satisfying 
\begin{equation} \label{etaeps}
\eta_\delta = 1 \text{ on } B_{r-\delta} \; , \quad |\nabla \eta_\delta|\leq \frac{C}{\delta}\; , \quad |D^2\eta_\delta|\leq \frac{C}{\delta^2}.
\end{equation}

We finally define
$$z_{\delta,n,k}:= \eta_\delta h_n + (1-\eta_\delta) w_{\e_k}=w_{\e_k}+\eta_\delta (h_n-w_{\e_k}).$$
Observe that $z_{\delta,n,k} \in H^2_{0,\Sigma_{\e_k}}(B_r)$ provided that $k \geq k_n$ is large enough. Consequently, since $z_{\delta,n,k} \in w_{\e_k} + H^2_0(B_r)$, we infer thanks to \eqref{eq:biharm} and Remark \ref{rem:airy-min} that
$$\int_{B_r} |D^2 w_{\e_k}|^2\, dx \leq \int_{B_r}|D^2 z_{\delta,n,k}|^2\, dx,$$
or still
%A direct computation yields
%\begin{eqnarray*}
%D^2 z_{\e,n,k}&=&D^2w_{\rho_k} + (h_n-w_{\rho_k})D^2\eta_\e +2\nabla \eta_\e \otimes (\nabla h_n-\nabla w_{\rho_k})+\eta_\e(D^2h_n-D^2w_{\rho_k}) \\
%&=& \eta_\e D^2h_n + (1-\eta_\e)D^2w_{\rho_k} + (h_n-w_{\rho_k})D^2\eta_\e +2\nabla \eta_\e \otimes (\nabla h_n-\nabla w_{\rho_k}),
%\end{eqnarray*}
%thus
\begin{multline*}
\int_{B_r} |D^2w_{\e_k}|^2\, dx \leq \int_{B_r}|\eta_\delta D^2h_n + (1-\eta_\delta)D^2w_{\e_k}|^2\, dx\\
+ \int_{B_r}|(h_n-w_{\e_k})D^2\eta_\delta +2\nabla \eta_\delta \otimes (\nabla h_n-\nabla w_{\e_k})|^2\, dx \\
+ 2\int_{B_r}\big[\eta_\delta D^2h_n + (1-\eta_\delta)D^2w_{\e_k} \big] :\big[(h_n-w_{\e_k})D^2\eta_\delta +2\nabla \eta_\delta \otimes (\nabla h_n-\nabla w_{\e_k})\big]\, dx  .  
\end{multline*}
By convexity, we get that
$$\int_{B_r}|\eta_\delta D^2h_n + (1-\eta_\delta)D^2w_{\e_k}|^2\, dx \leq  \int_{B_r}\eta_\delta| D^2h_n |^2\, dx + \int_{B_r}(1-\eta_\delta)|D^2w_{\e_k}|^2\, dx,$$
and thanks to \eqref{etaeps}
\begin{multline*}
\int_{B_r} \eta_\delta |D^2w_{\e_k}|^2\, dx \leq \int_{B_r}\eta_\delta |D^2h_n|^2\, dx  \\
+ C \int_{B_r\setminus B_{r-\delta}}\left(\frac{1}{\delta^{4}}|h_n-w_{\e_k}|^2 +\frac{1}{\delta^2} |\nabla h_n-\nabla w_{\e_k}|^2\right) dx  \\
+ 2\int_{B_r\setminus B_{r-\delta}}\big[\eta_\delta D^2h_n + (1-\eta_\delta)D^2w_{\e_k} \big] :\big[(h_n-w_{\e_k})D^2\eta_\delta +2\nabla \eta_\delta \otimes (\nabla h_n-\nabla w_{\e_k})\big]\, dx. 
\end{multline*}
Letting first $k \to \infty$ and then $n \to \infty$, using that $w_{\varepsilon_k}\to w_{\Sigma_0}$ in $H^1(B_r)$ and that $h_n \to w_{\Sigma_0}$ in $H^2(B_r)$, we obtain
$$\limsup_{k \to \infty} \int_{B_r} \eta_\delta |D^2 w_{\e_k}|^2\, dx \leq \int_{B_r} |D^2 w_{\Sigma_0}|^2\, dx.$$
On the other hand
$$\lim_{k \to \infty}\int_{B_r}(1-\eta_\delta) |D^2w_{\e_k}|^2\, dx = \int_{\overline{B_r}} (1-\eta_\delta)d\mu \leq \mu(\overline{B_r}\setminus B_{r-\delta}).$$
Therefore we can write that
\begin{eqnarray}
\limsup_{k\to \infty} \int_{B_r}  |D^2w_{\e_k}|^2 &\leq &\limsup_{k\to \infty} \int_{B_r} \eta_\delta |D^2w_{\e_k}|^2\, dx +\limsup_{k \to \infty}\int_{B_r}(1-\eta_\delta) |D^2w_{\e_k}|^2\, dx\nonumber\\
&\leq & \int_{B_r} |D^2 w_{\Sigma_0}|^2\, dx + \mu(\overline{B_r}\setminus B_{r-\delta}). \label{inequ001}
\end{eqnarray}
Finally, letting $\delta \to 0$ in \eqref{inequ001} and using the fact that $\mu(\partial B_r)=0$, we get the desired bound 
$$\limsup_{k\to \infty} \int_{B_r}  |D^2 w_{\e_k}|^2 \leq \int_{B_r} |D^2 w_{\Sigma_0}|^2\, dx,$$
%Gathering with \eqref{eq:liminf} yields
%$$\lim_{k\to \infty} \int_{B_r}  |D^2w_{\e_k}|^2 = \int_{B_r} |D^2 w_{\Sigma_0}|^2\, dx,$$
which ensures together with \eqref{eq:liminf} that $w_{\e_k}$ converges strongly to $w_{\Sigma_0}$ in $H^2(B_r)$.

\medskip

\noindent {\bf Biharmonicity. }In order to show that $w_{\Sigma_0}$ solves a biharmonic Dirichlet problem outside the crack $\Sigma_0$ is is enough to check that it satisfies the minimality property
$$\int_{B_R} |D^2 w_{\Sigma_0}|^2\, dx \leq \int_{B_R} |D^2 w|^2\, dx$$
for all $w \in  w_{\Sigma_0}+  H^2_0(B_R\setminus \Sigma_0)$. Let $z\in H^2_0(B_R \setminus \Sigma_0)$, by density, there exists a sequence of functions $(z_n) \subset \C^\infty_c(B_R \setminus \Sigma_0)$ such that $z_n \to z$ strongly in $H^2(B_R \setminus \Sigma_0)$. Since $z_n=0$ in a neighborhood of $\Sigma_0$, it follows by Hausdorff convergence that $z_n=0$ in a neighborhood of $\Sigma_{\e_k}$ for $k \geq k_n$ large enough, for some integer $k_n \in \N$. Therefore, for any $k \geq k_n$, $w_{\e_k} + z_n \in w_{\e_k} + H^2_{0,\Sigma_{\e_k}}(B_R)$ is an admissible competitor for the minimality property satisfied by the Airy function (see Remark \ref{rem:airy-min}), and 
$$\int_{B_R} |D^2 w_{\e_k}|^2\, dx \leq \int_{B_R} |D^2 w_{\e_k}+D^2 z_n|^2\, dx.$$
Letting first $k \to \infty$ and then $nÊ\to \infty$, and using the strong convergence of $(w_{\e_k})$ established before yields
$$\int_{B_R} |D^2 w_{\Sigma_0}|^2\, dx \leq \int_{B_R} |D^2 w_{\Sigma_0}+D^2 z|^2\, dx.$$
The proof of the Proposition is now complete.
\end{proof}

\begin{rem}\label{elliptic-reg}
By elliptic regularity, it follows that $w_{\Sigma_0}$ is smooth outside the origin up to both sides of $\Sigma_0$. In particular, for every $0<r<R<\infty$ and for every $k \in \N$, $w_{\Sigma_0} \in H^k((B_R \setminus B_r) \setminus \Sigma_0)$ and is a solution for the problem \eqref{prob1} in a stronger sense.
\end{rem}

It turns out that $w_{\Sigma_0}$ can be made explicit by showing that it is a positively $3/2$-homogeneous function.  %The remaining of this section is devoted to prove the following result.
The proof of this result follows an argument given by Monique Dauge, relying on the theory introduced by Kondrat'ev in \cite{kond67}, that is briefly recalled in Appendix \ref{K}.

\begin{prop}\label{eq:3/2homo}
The function $w_{\Sigma_0}$ is positively $3/2$-homogeneous. More precisely, in polar coordinates, we have for all $(r,\theta) \in (0,+\infty) \times (0,2\pi)$,
$$w_{\Sigma_0}(r\cos \theta,r \sin\theta)=r^{3/2} \left[c_1 \psi_1(\theta) + c_2 \psi_2(\theta) \right],$$
where $c_1$ and $c_2 \in \R$  are constants, while $\psi_1$ and $\psi_2$ are given by
\begin{eqnarray}
\psi_1(\theta)& :=& \left[\frac{3}{2}Ê\cos\left(\frac{\theta}{2} \right) - \frac{1}{2}Ê\cos\left(\frac{3\theta}{2} \right)\right],\label{defu1}\\
\psi_2(\theta) & :=& \left[\frac{3}{2}Ê\sin\left(\frac{\theta}{2} \right) + \frac{1}{2}Ê\sin\left(\frac{3\theta}{2} \right)\right].\label{defu2}
\end{eqnarray}
\end{prop}

\begin{proof}
Let $w_{\Sigma_0}$ be the biharmonic function in $\R^2\setminus \Sigma_0$ with homogeneous Dirichlet boundary conditions given by Proposition \ref{thm:blowup-airy}, and let $\chi\in \C_c^\infty(\R^2;[0,1])$ be a cut-off function satisfying $\chi=1$ in $B_1$ and $\chi=0$ in $\R^2\setminus B_2$. We decompose $w_{\Sigma_0}$ as follows:
$$w_{\Sigma_0}=w_0+w_\infty$$
where $w_0:= \chi w_{\Sigma_0}$ and $w_\infty:=(1-\chi)w_{\Sigma_0}$. Of course both $w_0$ and $w_\infty$ still satisfy homogenous boundary Dirichlet conditions on $\Sigma_0$, and one can check that
$$\Delta^2 w_0 =f_0 \text{ and } \Delta^2 w_\infty =f_\infty \; \text{ in } \R^2\setminus \Sigma_0,$$
for some $f_0$ and $f_\infty$ supported in the annulus $B_2\setminus B_1$. In addition, according to Remark \ref{elliptic-reg}, it follows that both $f_0$ and $f_\infty \in H^k(\R^2\setminus \Sigma_0)$ for every $k \in \N$, and consequently $f_0$ and $f_\infty \in V^\ell_\beta(\R^2\setminus \Sigma_0)$ for all $\ell \in \mathbb Z$ and all $\beta \in \R$ (we recall Section \ref{sectionKond} for the definition of $V^\ell_\beta$). We next intend to apply Theorem \ref{kondth} to $w_0$ and $w_\infty$ separately. 

\medskip

\noindent {\bf Step 1: Analysis of $w_0$. } Since $w_0\in H^2_0(\R^2\setminus \Sigma_0)$, we get that $w_0 \in V_0^2(\R^2 \setminus \Sigma_0)$. To establish this property, it suffices to check that the functions $x \mapsto |x|^{-1} \partial^\alpha w_0(x)$ (with $|\alpha|=1$) and $x \mapsto |x|^{-2} w_0(x)$  belong to $L^2(\R^2 \setminus \Sigma_0)$. Indeed,
%To prove \eqref{fact00y} we shall use the following Poincaré inequality in each corona  $A_R:=B(0,2R)\setminus B(0,R)$, 
%$$\int_{A_R} |f(x)|^2 dx \leq C_0 R^2 \int_{A_R} \|\nabla f\|^2 dx, $$
%valid for  functions $f\in H^1_0(\R^2 \setminus \Gamma)$. From this we deduce that, if $\gamma\geq 0$,
\begin{eqnarray*}
\int_{\R^2 \setminus \Sigma_0}Ê|x|^{-2} |\partial^\alpha w_0|^2\, dx & = & \sum_{j \in \mathbb{Z}} \int_{(B_{2^{j+1}} \setminus B_{2^{j}})\setminus \Sigma_0} |x|^{-2} |\partial^\alpha w_0|^2\, dx\\
& \leq &  \sum_{j \in \mathbb{Z}}  2^{-2j}\int_{(B_{2^{j+1}} \setminus B_{2^{j}})\setminus \Sigma_0} |\nabla w_0|^2\, dx.
\end{eqnarray*}
Since all weak derivatives $\partial^{\alpha} w_0$ for $|\alpha|=1$ belong to $H^1_{0,\Sigma_0}((B_{2^{j+1}}\setminus B_{2^j})\setminus \Sigma_0)$, Poincar\'e inequality  yields
$$\int_{(B_{2^{j+1}} \setminus B_{2^{j}})\setminus \Sigma_0} |\nabla w_0|^2\, dx \leq C_0 2^{2j} \int_{(B_{2^{j+1}} \setminus B_{2^{j}})\setminus \Sigma_0} |D^2 w_0|^2\, dx,$$
for some constant $C_0>0$ independent of $j$, and thus
\begin{equation}\label{1401}
\int_{\R^2 \setminus \Sigma_0}Ê|x|^{-2} |\partial^\alpha w_0|^2\, dx  \leq  C_0 \sum_{j \in \mathbb{Z}}  \int_{(B_{2^{j+1}} \setminus B_{2^{j}})\setminus \Sigma_0} |D^2 w_0|^2\, dx = C_0 \int_{\R^2\setminus \Sigma_0} |D^2 w_0|^2\, dx<\infty.
\end{equation}
Similarly, we have
\begin{eqnarray*}
\int_{\R^2 \setminus \Sigma_0}Ê|x|^{-4} |w_0|^2\, dx & = & \sum_{j \in \mathbb{Z}} \int_{(B_{2^{j+1}} \setminus B_{2^{j}})\setminus \Sigma_0} |x|^{-4} |w_0|^2\, dx\\
& \leq &  \sum_{j \in \mathbb{Z}}  2^{-4j}\int_{(B_{2^{j+1}} \setminus B_{2^{j}})\setminus \Sigma_0} |w_0|^2\, dx.
\end{eqnarray*}
Applying again Poincar\'e inequality to the function $w_0 \in H^1_{0,\Sigma_0}((B_{2^{j+1}}\setminus B_{2^j})\setminus \Sigma_0)$, we obtain
$$\int_{(B_{2^{j+1}} \setminus B_{2^{j}})\setminus \Sigma_0} |w_0|^2\, dx \leq C_0 2^{2j} \int_{(B_{2^{j+1}} \setminus B_{2^{j}})\setminus \Sigma_0} |\nabla w_0|^2\, dx,$$
and thus, according to \eqref{1401},
$$\int_{\R^2 \setminus \Sigma_0}Ê|x|^{-4} |w_0|^2\, dx  \leq  C_0 \sum_{j \in \mathbb{Z}} 2^{-2j} \int_{(B_{2^{j+1}} \setminus B_{2^{j}})\setminus \Sigma_0} |\nabla w_0|^2\, dx \leq4 C_0  \int_{\R^2\setminus \Sigma_0} |x|^{-2}|\nabla w_0|^2\, dx<\infty.$$

%
%
%
%\begin{eqnarray}
%&\leq &  \sum_{j\in \mathbb{Z}}2^{2\gamma j}  C_0 2^{2j} \int_{A_{2^j}} \|\nabla u\|^2 \notag \\
%&\leq & C_0 2^{2\gamma+2}\sum_{j\in \mathbb{Z}}2^{2\gamma j}   2^{(2\gamma+2)(j-1)} \int_{A_{2^j}} r^{2\gamma+2} \|\nabla u\|^2  \notag \\
%&\leq & C_0 \sum_{j\in \mathbb{Z}} \int_{A_{2^j}} r^{2\gamma+2} \|\nabla u\|^2 =C_0  \int_{\R^2\setminus \Gamma} r^{2\gamma+2} \|\nabla u\|^2 < +\infty. \notag 
%\end{eqnarray}
%and a similar argument for $\gamma\leq 0$ gives the same inequality with a different constant, which ends the proof.

Since in particular $f_0 \in V_{\beta}^{-2}(\R^2 \setminus \Sigma_0) \cap V_{0}^{-2}(\R^2 \setminus \Sigma_0)$ for any $\beta<0$, applying Theorem \ref{kondth} yields that for any $\beta \in \R^{-}\setminus \mathcal{S}$, there exists $z_0 \in V^2_{\beta}(\R^2 \setminus \Sigma_0)$ such that
$$w_0= z_0+ \sum_{\lambda \in \mathcal{S}\cap (1,  1-\beta) } r^{\lambda} \varphi_\lambda(\theta).$$

\medskip

\noindent {\bf Step 2: Analysis of $w_\infty$. } We first observe that the growth condition \eqref{eq:bound-u_sigma_0} satisfied by $w_{\Sigma_0}$ shows that 
$$\sup_{R>0}\frac{1}{R}\int_{B_R} |D^2 w_\infty|^2\, dx <\infty$$
since $w_\infty$ is supported in $\R^2 \setminus B_1$. Let us check that this growth condition implies $w_\infty \in V^2_\beta(\R^2\setminus \Sigma_0)$ with $\beta<-1/2$.
Indeed, for $|\alpha|=2$,
\begin{eqnarray}
\int_{\R^2\setminus \Sigma_0} |x|^{2\beta} |\partial^\alpha w_\infty|^2 dx &=& \int_{\R^2\setminus  (\Sigma_0 \cup B_1)} |x|^{2\beta} |\partial^\alpha w_\infty|^2 dx  \notag\\
&\leq  & \sum_{j\geq 0}  \int_{B_{2^{j+1}}\setminus B_{2^j}} |x|^{2\beta} |D^2w_\infty|^2 dx \notag \\
 &\leq&   \sum_{j\geq 0}  2^{2\beta j}\int_{B_{2^{j+1}}\setminus B_{2^j}}  |D^2w_\infty|^2 dx \notag \\
&\leq&   \sum_{j\geq 0} C 2^{2\beta j} 2^{j+1}<+\infty \notag  
\end{eqnarray}
provided that $\beta <-1/2$. We next show that the functions $x \mapsto |x|^{\beta-1} \nabla w_\infty(x)$ and $x \mapsto |x|^{\beta-2} w_\infty(x)$ belong to $L^2(\R^2 \setminus \Sigma_0)$ arguing exactly as in Step 1. It again relies on a dyadic partition of $\R^2 \setminus B_1$ together with the following Poincar\'e inequalities in each annuli $B_{2^{j+1}}\setminus B_{2^j}$ 
$$\int_{(B_{2^{j+1}}\setminus B_{2^j}) \setminus \Sigma_0} |w_\infty|^2 dx \leq C_0 2^{2j} \int_{(B_{2^{j+1}}\setminus B_{2^j}) \setminus \Sigma_0} |\nabla w_\infty|^2 dx, $$
and
$$\int_{(B_{2^{j+1}}\setminus B_{2^j})\setminus \Sigma_0} |\nabla w_\infty|^2 dx \leq C_0 2^{2j} \int_{(B_{2^{j+1}}\setminus B_{2^j})\Sigma_0} |D^2 w_\infty|^2 dx,$$
which hold since both $w_\infty$ and $\nabla w_\infty$ vanish on $\Sigma_0$ allowing us to apply Poincar\'e inequality  to them. Therefore it leads to $w_\infty\in  V^2_{\beta_0}(\R^2 \setminus \Sigma_0)$ for $\beta_0=-1/2-\e$, where $\e>0$ is small.

\medskip

\noindent {\bf Conclusion. } We finally gather all the results established so far by taking the same $\beta_0$ for the above functions $w_0$  and $w_\infty$. Observing that $\mathcal{S}\cap (1,1-\beta_0) =\{3/2\}$, we get that, in polar coordinates,
$$w(r\cos \theta,r \sin \theta)= r^{3/2} \varphi_{3/2}(\theta)+z(r\cos \theta,r \sin \theta) \quad \text{ for a.e. }(r,\theta) \in (0,+\infty) \times (0,2 \pi),$$
for some $z \in V_{-1/2-\varepsilon}^2(\R^2 \setminus \Sigma_0)$. We finally complete the proof of the proposition by establishing that $z=0$. To this aim, we recall that the function $(r,\theta) \mapsto r^{3/2}\phi_{3/2}(\theta)$ is biharmonic on $\R^2 \setminus \Sigma_0$, and that it vanishes together with its gradient on the crack $\Sigma_0$. In other words it is a solution of $(P_1)$ with $f=0$. We deduce that  $z \in V_{-1/2-\varepsilon}^2(\R^2 \setminus \Sigma_0)$ must be a solution of $(P_1)$ with $f=0$ as well. But since $-1/2-\varepsilon \not \in \mathcal{S}$, Theorem \ref{kondth0} (with $\beta=3/2$ and $\ell=2$) ensures that $z=0$.
\end{proof}

%%%%%%%%%%%%%%%%%%%%
\subsection{Blow-up analysis of the displacement}

We are now in position to study the blow-up of the displacement. We show that, up to a subsequence and rigid movement, it converges to the usual positively $1/2$-homogeneous function satisfying the Lam\'e system outside a half-line.

\begin{thm} \label{thm:blowup} 
For every sequence $(\e_n) \searrow 0^+$, there exist a subsequence $(\e_k)Ê\equiv (\e_{n_k})Ê\searrow 0^+$, a sequence $(m_k)$ of rigid movements and a function $u_{\Sigma_0} \in LD_{\rm loc}(\R^2 \setminus \Sigma_0)$ such that the blow-up sequence of displacements satisfies
\begin{equation}\label{eq:ukbis}
\begin{cases}
u_{\e_k}-m_k \to  u_{\Sigma_0} \quad \text{ strongly in }L^2_{\rm loc}(\R^2;\R^2),\\
e(u_{\e_k})\chi_{\R^2 \setminus \Sigma_{\e_k}} \to e(u_{\Sigma_0}) \quad \text{ strongly in }L^2_{\rm loc}(\R^2;\Ms).
\end{cases}
\end{equation}
In addition, the function $u_{\Sigma_0}$ is positively $1/2$-homogeneous and it is given in polar coordinates by
\begin{equation}\label{cracktip007}
u_{\Sigma_0}(r\cos \theta,r \sin \theta)=\sqrt r [\kappa_1\phi_1(\theta) + \kappa_2\phi_2(\theta)] \quad \text{ for all }(r,\theta) \in (0,+\infty) \times (0,2\pi),
\end{equation}
where $\kappa_1$ and $\kappa_2 \in \R$ are constants, while $\phi_1$ and $\phi_2$ are defined by 
\begin{eqnarray}
\phi_1(\theta):= \left(
\begin{array}{l}
\frac{\lambda+\mu}{2}\cos\left(\frac{3\theta}{2}\right) + \frac{\lambda-3\mu}{2}\cos\left(\frac{\theta}{2}\right)\\
\frac{\lambda+\mu}{2}\sin\left(\frac{3\theta}{2}\right) + \frac{5\lambda+9\mu}{2}\sin\left(\frac{\theta}{2}\right)
\end{array}
\right), \label{defPHI1}
\end{eqnarray}
and
\begin{eqnarray}
\phi_2(\theta):= \left(
\begin{array}{l}
-\frac{\lambda+\mu}{2}\sin\left(\frac{3\theta}{2}\right) - \frac{3\lambda+7\mu}{2}\sin\left(\frac{\theta}{2}\right)\\
\frac{\lambda+\mu}{2}\cos\left(\frac{3\theta}{2}\right) + \frac{\lambda+5\mu}{2}\cos\left(\frac{\theta}{2}\right)
\end{array}
\right). \label{defPHI2}
\end{eqnarray}
\end{thm}

\begin{proof} A scalar version of that theorem is contained in \cite[Theorem 1.1]{CL}, but the proof does not extend directly to the vectorial case. This is why we present here an alternative argument based on the Airy function. 

Let $(\e_k)$ be the subsequence given by Proposition \ref{thm:blowup-airy}. As in the proof of that result, $R>1$ is fixed, and $k \in \N$ is large enough such that $2R<R_0/\e_k$.

\medskip

\noindent {\bf Compactness. } Let us denote by $C:=B_{1/4}(1/2,0)$ the ball of center $(1/2,0)$ and radius $1/4$. We consider the following sequence of rigid displacements
$$\bar u_k(x):=\frac{1}{|C|}\int_{C}u_{\e_k}(y)\,  dy + \left(\frac{1}{|C|}\int_{C}\frac{\nabla u_{\e_k}(y)- \nabla u_{\e_k}(y)^T}{2}\, dy\right)\left(x- \frac{1}{|C|}\int_{C}y\, dy\right).$$
%\begin{equation}\label{eq:uk}
%u_k(x):=u_{\e_k}(x)-m_k(x).
%\end{equation}
%Since $m_k$ is a rigid movement, the elastic strain satisfies
%$$e(u_k)=e(u_{\e_k}),$$
Thanks to \eqref{eq:hessian}, \eqref{eq:displacement-rescaled} and \eqref{eq:airy-rescaled}, the stress is given by 
\begin{equation}\label{eq:uw}
\CC e(u_{\e_k})=
\left(
\begin{array}{cc}
D_{22} w_{\e_k} & -D_{12} w_{\e_k}\\
-D_{12} w_{\e_k} & D_{11} w_{\e_k}
\end{array}
\right).
\end{equation}
Therefore, according to \eqref{eq:airy-bound}, we deduce that the sequence $(e(u_{\e_k}))_{k \in \N}$ is uniformly bounded in $L^2(B_R;\Ms)$. Consequently, up to a subsequence (not relabeled), there exists $e \in L^2_{\rm loc}(\R^2;\Ms)$ such that $e(u_{\e_k}) \wto e$ weakly in $L^2_{\rm loc}(\R^2;\Ms)$. In addition, the strong $H^2_{\rm loc}(\R^2)$-convergence of the Airy function established in Theorem \ref{thm:blowup-airy} together with \eqref{eq:hessian}, \eqref{eq:displacement-rescaled} and \eqref{eq:airy-rescaled} shows that actually $e(u_{\e_k}) \to e$ strongly in $L^2_{\rm loc}(\R^2;\Ms)$.

We next show that $e$ is  the symmetrized gradient of some displacement. To this aim, we consider, for any $0<\delta<1/10$, the Lipschitz domain
$$U_\delta:=\{x\in B_R : {\rm dist}(x,\Sigma_0)>\delta\}.$$
Note that for such $\delta$, $C \subset\subset U_\delta$, while  $\Sigma_{\e_k} \cap U_\delta = \emptyset$ for $k$ large enough (depending on $\delta$). By virtue of the Poincar\'e-Korn inequality \cite[Theorem 5.2 and Example 5.3]{AMR} we get that 
\begin{equation}\label{eq:PK}
\|u_{\e_k}-\bar u_k\|_{H^1(U_\delta)}\leq c_\delta \| e(u_{\e_k})\|_{L^2(U_\delta)},
\end{equation}
for some constant $c_\delta>0$ depending on $\delta$. Thanks to a diagonalisation argument, we obtain for a subsequence (not relabeled) a function $\hat u_{\Sigma_0} \in LD_{\rm loc}(\R^2\setminus \Sigma_0)$ such that $u_{\e_k}-\bar u_k \to \hat u_{\Sigma_0}$ strongly in $H^1(U_\delta;\R^2)$, for any $0<\delta<1/10$. Necessarily we must have that $e=e(\hat u_{\Sigma_0})$ and 
$$\begin{cases}
u_{\e_k}-\bar u_k \to  \hat u_{\Sigma_0} \quad \text{ strongly in }L^2_{\rm loc}(\R^2;\R^2),\\
e(u_{\e_k})\chi_{\R^2 \setminus \Sigma_{\e_k}} \to e(\hat u_{\Sigma_0}) \quad \text{ strongly in }L^2_{\rm loc}(\R^2;\Ms).
\end{cases}$$

%Moreover, passing to the limit in \eqref{eq:uw} and \eqref{eq:PK}, and using Theorem \ref{thm:blowup-airy} yields
%\begin{equation}\label{eq:hatuw}
%\CC e(\hat u_{\Sigma_0})=
%\left(
%\begin{array}{cc}
%D_{22} \hat w_{\Sigma_0} & -D_{12} \hat w_{\Sigma_0}\\
%-D_{12} \hat w_{\Sigma_0} & D_{11} \hat w_{\Sigma_0}
%\end{array}
%\right).
%\end{equation}
%and
%\begin{equation}\label{eq:hatPK}
%\|\hat u_{\Sigma_0}\|_{H^1(U_\delta)}\leq C_\delta \| e(\hat u_{\Sigma_0})\|_{L^2(U_\delta)}.
%\end{equation}

\medskip

\noindent {\bf Minimality. } We next show that $\hat u_{\Sigma_0}$ satisfies the minimality property
$$\int_{B_R} \CC e(\hat u_{\Sigma_0}):e(\hat u_{\Sigma_0})\, dx \leq \int_{B_R} \CC e(\hat u_{\Sigma_0}+v):e(\hat u_{\Sigma_0}+v)\, dx$$
for all $v \in LD(B_R \setminus \Sigma_0)$ such that $v=0$ on $\partial B_R \setminus \Sigma_0$. According to \cite[Theorem 1]{C}, it is enough to consider competitors $v \in H^1(B_R \setminus \Sigma_0;\R^2)$ with $v=0$ on $\partial B_R \setminus \Sigma_0$. Moreover, since $\{0\}$ has zero $\Cap_{1,2}$-capacity, we can also assume without loss of generality that $v=0$ in a neighborhood of the origin.

Denoting by $C_k^\pm$ the connected component of $(B_R \setminus \Sigma_{\e_k}) \cap \{x_1 \leq 0\}$ which contains the point $(-1/2,\pm1/2)$, we define $v_k$ as follows:
\begin{itemize}
\item $v_k(x_1,x_2)=v(x_1,x_2)$ if $(x_1,x_2) \in [B_R \cap \{x_1>0\}] \cup [C_k^+ \cap \{x_2 \geq 0\}] \cup [ C_k^- \cap \{x_2 \leq 0\}]$;
%\item $v_k(x_1,x_2)=v(x_1,x_2)$ if $(x_1,x_2) \in C_k^+ \cap \{x_2 \geq 0\}$;
\item $v_k(x_1,x_2)=v(x_1,-x_2)$ if $(x_1,x_2) \in [C_k^+ \cap \{x_2 < 0\}] \cup [ C_k^- \cap \{x_2 >0\}]$;
%\item $v_k(x_1,x_2)=v(x_1,x_2)$ if $(x_1,x_2) \in C_k^- \cap \{x_2 \leq 0\}$;
%\item $v_k(x_1,x_2)=v(x_1,-x_2)$ if $(x_1,x_2) \in C_k^- \cap \{x_2 >0\}$;
\item $v_k(x_1,x_2)=0$ elsewhere.
\end{itemize}
Then, one can check that $v_k \in H^1(B_R \setminus \Sigma_{\e_k};\R^2)$ and $v_k=0$ on $\partial B_R \setminus \Sigma_{\e_k}$. Moreover, $v_k \to v$ strongly in $L^2(B_R;\R^2)$ and $(\nabla v_k)\chi_{B_R \setminus \Sigma_{\e_k}} \to \nabla v$ strongly in $L^2(B_R;\Ms)$. Therefore, thanks to the minimality property \eqref{minprob1} satisfied by $u_0$, we infer that
$$\int_{B_R} \CC e(u_{\e_k}):e(u_{\e_k})\, dx \leq \int_{B_R} \CC e(u_{\e_k}+v_k):e(u_{\e_k}+v_k)\, dx,$$
so that passing to the limit as $k \to \infty$, and invoking the strong convergences \eqref{eq:ukbis} yields the desired minimality property.
%$$\int_{B_R} \CC e(\hat u_{\Sigma_0}):e(\hat u_{\Sigma_0})\, dx \leq \int_{B_R} \CC e(\hat u_{\Sigma_0}+v):e(\hat u_{\Sigma_0}+v)\, dx.$$

\medskip

\noindent {\bf Explicit expression of the displacement. } According to Theorem I and  Remark 1.2 in \cite{Gr2}, (see also Remark 2.1. in \cite{CFM}), there exist constants $\kappa_1$ and $\kappa_2 \in \R$, and a function $g \in H^2_{\rm loc}(\R^2)$ such that
$$\hat u_{\Sigma_0}(r\cos\theta,r\sin\theta)=\sqrt r[\kappa_1\phi_1(\theta) + \kappa_2 \phi_2(\theta)] + g(r\cos\theta,r\sin\theta)  \text{ for a.e. }(r,\theta) \in (0,+\infty) \times (0,2\pi).$$ 
The previous expression of the displacement shows that 
\begin{equation}\label{1408}
\CC e(\hat u_{\Sigma_0})=\Phi + \CC e(g),
\end{equation}
where $\Phi$ is a positively $-1/2$-homogeneous function. On the other hand,  passing to the limit in \eqref{eq:uw} as $k \to \infty$ and using Proposition \ref{thm:blowup-airy} yields
\begin{equation}\label{1409}
\CC e(\hat u_{\Sigma_0})=\left(
\begin{array}{ll}
D_{22} w_{\Sigma_0} & -D_{12} w_{\Sigma_0}\\
-D_{12} w_{\Sigma_0} & D_{11} w_{\Sigma_0}
\end{array}
\right).
\end{equation}
According to Proposition \ref{eq:3/2homo} the right hand side of the previous equality is positively $-1/2$-homogeneous as well. Therefore gathering \eqref{1408} and \eqref{1409} ensures that $e(g)=0$ which shows that $g=m$ is a rigid movement. We finally define the rigid displacement $m_k:=\bar u_k+m$ which fullfills the conclusions of the proposition.
\end{proof}

\section{Energy release rate}\label{sec:err}

\noindent Following the approach  of \cite{CFM}, our aim is to give a definition of energy release rate by studying the convergence of the blow-up functional $\frac{1}{\e}\G(\e \Gamma)$.  The following statement is the same as \cite[Theorem 3.1.]{CFM}, but with the substantial difference that now $\Gamma_0$ is not assumed to be a straight line segment near the origin, but only blowing-up to such a segment for the Hausdorff distance.

\begin{thm}\label{main2}
Let $(\Gamma_\e)_{\e>0}$ be a sequence of crack increment in $\K(\overline{\O})$ be such that $\sup_\e \HH^1(\Gamma_\e)<\infty$, and $\Gamma_\e \to \Gamma$ in the sense  of Hausdorff in $\overline\O$. Let us consider the rescaled crack $\Sigma_\e$ and displacement $u_\e$ defined, respectively by \eqref{eq:crack-rescaled} and \eqref{eq:displacement-rescaled}. %Assume that $\frac{1}{\e}(\Gamma_0\cap B_\e)$ converges for the Hausdorff distance in the unit ball as $\e \to 0$, to the left hand side radius  $\Sigma_0$. 
Then for every sequence $(\e_n) \searrow 0^+$, there exist a subsequence $(\e_k)Ê\equiv (\e_{n_k})Ê\searrow 0^+$ and a rotation $\mathcal R \in SO(2)$ such that
% along which $u_j - \bar u_j$ converges to $u(0)+u_{\Sigma_0}$, as given by . Then recalling the definition of $\G_\e$ in  \eqref{defGcal}, we have that
\begin{eqnarray}
\lim_{k \to \infty} \frac{1}{\e_k}\G(\e_k \Gamma_{\e_k}) = \F(\Gamma) \label{ERRate}
\end{eqnarray}
where $\F$ is defined by
\begin{multline}\label{eq:limit-pb}
\F(\Gamma):= \min_{w\in LD(\R^2 \setminus (\Sigma_0\cup \mathcal R(\Gamma)))} \Big\{ \frac{1}{2} \int_{\R^2} \CC e(w):e(w)\, dx + \int_{B_R}\CC e(u_{\Sigma_0}):e(w)\, dx\\
 - \int_{\partial B_R} \CC e(u_{\Sigma_0}) : (w \odot \nu) d \HH^1 \Big\},
\end{multline}
where $R>0$ is  any radius such that $\Gamma \subset B_R$.
\end{thm}

\begin{rem} The proof of Theorem \ref{main2} follows the scheme of \cite[Theorem 3.1]{CFM}, but some technical issues arise at two main points: 1)  the explicit expression for the blow-up at the origin  does not come directly from the literature but now follows from our first main result Theorem \ref{thm:blowup},  and 2) the construction of a recovery sequence of functions in the moving domains that converges in a strong sense to prove the minimality of the limit is more involved, since now after rescaling everything in $B_1$ our sequence of domains also moves on $\partial B_1$.
\end{rem}

\begin{rem} In the scalar case (antiplane) the limit does actually  not depend on the subsequence due to the existence of blow-up limit for the whole sequence \cite{CL}.
\end{rem}

\begin{proof}[Proof of Theorem \ref{main2}] Let $(\e_n) \searrow 0^+$ and $(\e_k)Ê\equiv (\e_{n_k})Ê\subset (\e_n)$ be the subsequence given by Theorem \ref{thm:blowup}. Let us consider the rotation $\mathcal R_\e$ be introduced at the beginning of section \ref{sec:blow-up}. It is not restrictive to assume that $\mathcal R_{\e_k}$ converges to some limit rotation $\mathcal R$. In particular $\mathcal R_{\e_k}(\Gamma_{\e_k})$ converges to $\mathcal R(\Gamma)$ in the sense of Hausdorff.

%\vspace{0.5cm} 
\noindent {\bf Rescaling. }We denote by $u_k$ a solution of the minimization problem
\begin{equation}
\min\left\{ \frac{1}{2}\int_{\Omega} \CC e(v):e(v) \, dx \; : \; v \in LD(\Omega \setminus (\Gamma_0\cup \e_k \Gamma_{\e_k}))\text{ and } v=\psi \text{ on } \partial \Omega \setminus (\Gamma_0 \cup \e_k \Gamma_{\e_k}) \right\}. \label{problemF}
\end{equation}
Recalling \eqref{defGcal} and \eqref{defG}, we can write 
$$G_{\e_k}=\frac{1}{\e_k}\G(\e_k \Gamma_{\e_k}) =\frac{1}{2\e_k}\int_\Omega \big[\CC e(u_k):e(u_k)-\CC e(u_0):e(u_0)\big]\, dx,$$
and setting $\hat w_k:= u_k-u_0$, we obtain that
$$\frac{1}{\e_k}\G(\e_k \Gamma_{\e_k})=\frac{1}{2\e_k}\int_{\Omega}\CC e(\hat w_k):e(\hat w_k)\, dx + \frac{1}{\e_k} \int_{\Omega} \CC e(\hat w_k):e(u_0)\, dx.$$
Since $\hat w_k= 0$ on $\partial \Omega \setminus (\Gamma_0 \cup \e_k \Gamma_{\e_k}) $, the variational formulation of \eqref{problemF} ensures that
$$\int_{\Omega} \CC e(u_k):e(\hat w_k)\, dx =0,$$ 
and it follows, writing $u_0=u_k-\hat w_k$,
\begin{equation}\label{1121}
\frac{1}{\e_k}\G(\e_k \Gamma_{\e_k}) =  -\frac{1}{2\e_k}\int_{\Omega} \CC e(\hat w_k) : e(\hat w_k)\, dx.
\end{equation}
On the other hand, from \eqref{problemF}  it is easy to see that $\frac{1}{\e_k}\G(\e_k \Gamma_{\e_k})$ is also resulting from a minimization problem with homogeneous boundary condition. Indeed, for any $\hat w \in LD(\Omega\setminus (\Gamma_0\cup \e_k \Gamma_{\e_k}))$ with $\hat w =0$ on $\partial \Omega \setminus  (\Gamma_0\cup \e_k \Gamma_{\e_k})$, denoting $v=u_0+\hat w$, we obtain that
$$\frac{1}{2}\int_{\Omega}\CC e(v):e(v) \, dx=\frac{1}{2}\int_{\Omega}\CC e(u_0):e(u_0)\, dx+\frac{1}{2}\int_{\Omega}\CC e(\hat w):e(\hat w)\, dx+\int_{\Omega}\CC e(u_0):e(\hat w)\, dx,$$
which implies 
\begin{eqnarray}
\frac{1}{\e_k}\G(\e_k \Gamma_{\e_k}) &=& \frac{1}{\e_k}\min\Big\{ \frac{1}{2}\int_{\Omega} \CC e(\hat w):e(\hat w) \, dx + \int_{\Omega} \CC e(u_0):e(\hat w) \, dx\; : \notag\\ 
 & & \hspace{1,5cm} \hat w \in LD(\Omega \setminus (\Gamma_0\cup \e_k \Gamma_{\e_k}))\text{ and } \hat w=0 \text{ on } \partial \Omega \setminus (\Gamma_0\cup \e_k \Gamma_{\e_k}) \Big\}\label{JF2} \\
 &=&  \frac{1}{2\e_k} \int_{\Omega} \CC e(\hat w_k):e(\hat w_k)\, dx + \frac{1}{\e_k}\int_{\Omega} \CC e(u_0):e(\hat w_k)\,  dx. \notag 
\end{eqnarray}

According to the assumptions done on $\Gamma_\e$, there exists $R>0$ such that if $\e$ is small enough, then $\Gamma_{\e}\subset \overline{B}_R \subset \Omega$, and $\mathcal{H}^1(\e \Gamma_\e)\leq C\e$ for some constant $C>0$ independent of $\e$. In addition, thanks to the lower bound in Proposition \ref{PROPbound}, we get again for $\e$ small enough,
$$- \frac{1}{\e}\G(\e\Gamma_{\e}) \leq C,$$
which implies from \eqref{1121}
\begin{eqnarray}
\frac{1}{\e_k}\int_{\Omega}\CC e(\hat w_k):e(\hat w_k)\, dx \leq C. \label{JF1}
\end{eqnarray}

We now proceed to the following change of variable:%$y= x/\e \mathcal R_\e(x) \in \frac{1}{\varepsilon_k}\Omega $ leading to the definition of
$$\Omega_k := \e_k^{-1}\mathcal R_{\e_k}(\Omega), \quad \Sigma_{\e_k}:=\e_k^{-1} \mathcal R_{\e_k} (\Gamma_0),$$
and for $y \in \O_k$,
$$w_k (y):=\e_k^{-1/2} \hat w_k(\mathcal R_{\e_k}^{-1}(\e_k y)), \quad u_{\e_k}(y):=\e_k^{-1/2} u_0(\mathcal R_{\e_k}^{-1}(\e_k y)).$$
We easily deduce from \eqref{JF1} that
\begin{eqnarray}
\int_{\Omega_k}\CC e(w_k):e(w_k)\, dx \leq C. \label{JF3}
\end{eqnarray}
We can also recast the minimisation problem in \eqref{JF2} in terms of $w_k$, which now writes as
 \begin{eqnarray}
\frac{1}{\e_k}\G(\e_k \Gamma_{\e_k}) &=& \min\Big\{ \frac{1}{2}\int_{\Omega_k} \CC e( w):e( w)\,  dx + \int_{\Omega_k} \CC e(u_{\e_k}):e( w)\, dx\; : \notag\\ 
 & & \hspace{1,5cm} w \in LD(\Omega_k \setminus (\Sigma_{\e_k} \cup \mathcal R_{\e_k}(\Gamma_{\e_k}) ))\text{ and } w=0 \text{ on } \partial \Omega_k \setminus (\Sigma_{\e_k} \cup  \mathcal R_{\e_k}(\Gamma_{\e_k}) ) \Big\}\notag \\
 &=&  \frac{1}{2} \int_{\Omega_k} \CC e(w_k):e(w_k) \, dx + \int_{\Omega_k} \CC e(u_{\e_k}):e(w_k)\,  dx\label{JFF}
% &=& -\frac12 \int_{\Omega_k}\CC e(w_k):e(w_k)\, dx,
  \end{eqnarray}
where we used \eqref{1121} in the last equality.

%%%%%%%%%%%%%%%%%%%%%%%%
\medskip

\noindent {\bf Compactness. } We now extend $w_k$ by $0$ outside $\Omega_k$ in such a way that  $w_k \in LD(\R^2\setminus  ( \Sigma_{\e_k} \cup  \mathcal R_{\e_k}(\Gamma_{\e_k}) ))$. Defining
$$
e_k:=\left\{
\begin{array}{ll}
e(w_k) & \text{in } \Omega_k \\
0 & \text{otherwise,}
\end{array}
\right.
$$
and using \eqref{JF3} together with the coercivity of $\CC$, we infer that the sequence $(e_k)_{k \in \N}$ is uniformy bounded in $L^2(\R^2;\Ms)$. 
%$$\|e_\|_{L^2(\R^2)}\leq C.$$
Consequently, up to a new subsequence (not relabeled), we can assume that $e_k \wto e$ weakly in $L^2(\R^2;\Ms)$ for some function $e\in L^2(\R^2;\Ms)$.

Let us recall that $\Sigma_{\e}\to \Sigma_0:=(-\infty,0] \times \{0\}$ locally in the sense of Hausdorff in $\R^2$, and that $\Gamma_{\e}\to \Gamma$ in the sense of Hausdorff in $\overline \O$. Let us denote by $\hat B:=B_{1/2}((R+1,0))$ the ball of $\R^2$ centered at the point $(R+1,0)$ and of radius $1/2$. Since $\Gamma \subset B_R$ and thus $\mathcal R(\Gamma) \subset B_R$, we deduce that $(\Sigma_0\cup \mathcal R(\Gamma))Ê\cap \hat B =\emptyset$. Therefore, for $k$ large enough, $\hat B \subset \O_k \setminus (\Sigma_{\e_k}Ê\cup \mathcal R_{\e_k}(\Gamma_{\e_k}) )$. Let us consider  a bounded and smooth open set $U\subset \R^2\setminus (\Sigma_0\cup  \mathcal R(\Gamma))$ containing $\hat B$. Then for all $k$ large enough, we have $\overline{U} \subset  \O_k \setminus (\Sigma_{\e_k}Ê\cup \mathcal R_{\e_k}(\Gamma_{\e_k}))$, and we denote  by $r_k$ the rigid movement defined by
$$r_k(x):= \frac{1}{|\hat B|}\int_{\hat B} w_k(y) \, dy +\left(\frac{1}{|\hat B|}\int_{\hat B}\frac{\nabla w_k(y) - \nabla w_k(y)^T}{2} \, dy\right)\left(x-\frac{1}{|\hat B|}\int_{\hat B} y \,dy\right).$$
 By Korn's inequality, we obtain that 
 $$\|w_k-r_k\|_{H^1(U)}\leq C_U,$$
 for some constant $C_U>0$ depending on $U$ but independent of $k$. This implies that, up to a subsequence, $w_k-r_k \wto w$ weakly in $H^1(U;\R^2)$ for some $w \in H^1(U;\R^2)$. By exhausting $\R^2\setminus (\Sigma_0\cup  \mathcal R(\Gamma))$ with countably many open sets, extracting successively many subsequences and using a diagonal argument, we obtain that $w \in H^1_{\rm loc}(\R^2\setminus (\Sigma_0\cup  \mathcal R(\Gamma));\R^2)$ and
 $$w_k-r_k \wto w \text{ weakly in } H^1_{\rm loc}(\R^2\setminus (\Sigma_0\cup  \mathcal R(\Gamma));\R^2).$$
 Moreover by uniqueness of the limit we infer that $e(w)=e$ a.e. in $\R^2\setminus (\Sigma_0\cup  \mathcal R(\Gamma))$, therefore that $e(w)\in L^2(\R^2;\Ms)$ and $w \in LD(\R^2\setminus (\Sigma_0\cup  \mathcal R(\Gamma)))$.

 %%%%%%%%%%%%%%%%%%%%%%%%%%%%%%%%
\medskip

\noindent {\bf Lower bound inequality.} Let $\zeta \in W^{1,\infty}(\R^2 ; [0,1])$ be a cut-off function such that $\zeta=1$ on $B_R$ and $\zeta=0$ on $\R^2\setminus B_{R'}$ for some given $R'>R$. Recalling \eqref{JFF}   we can write 
\begin{multline*}
 \frac{1}{\e_k}\G(\e_k \Gamma_{\e_k}) =\frac{1}{2} \int_{\Omega_k}  \CC e(w_k):e(w_k) \,dx\\
+ \int_{B_{R'}} \zeta\; \CC e(u_{\e_k}):e(w_k) \, dx+ \int_{\Omega_k\setminus B_R}(1-\zeta) \CC e(u_{\e_k}):e(w_k)\, dx\\
=\frac{1}{2} \int_{\Omega_k}  \CC e(w_k-r_k):e(w_k-r_k) \,dx\\
+ \int_{B_{R'}} \zeta\; \CC e(u_{\e_k}):e(w_k-r_k) \, dx+ \int_{\Omega_k\setminus B_R}(1-\zeta) \CC e(u_{\e_k}):e(w_k-r_k)\, dx.
 \end{multline*}
Let $R''<R$ be such that $\Gamma \subset B_{R''}$ and $\e_k R'' \not\in \mathcal N$, where $\mathcal N$ is the $\LL^1$-negligible set given by Lemma \ref{lem:IPP}. According to that result, we infer that
$$\int_{\O_k \setminus B_{R''}}  \CC e(u_{\e_k}) : e \big((1-\zeta) (w_k-r_k) \big)\, dx = -\int_{\partial B_{R''} \setminus \Sigma_{\e_k}} (1-\zeta) (\CC e(u_{\e_k}\nu) \cdot (w_k-r_k)\, d\HH^1=0,$$
and thus
%Since $\CC e(u_0^k).\nu=0$ (weakly) on $\Gamma_0^k$, $w_k=0$ on $\partial \Omega_k \setminus \Gamma_0^k$, and $(1-\zeta)=0$ on $\partial B_{R}$, if we integrate by parts it simply comes
% \begin{eqnarray}
% \int_{\Omega_k\setminus B_R}(1-\zeta) \CC e(u_0^k):e(w_k) \;dx 
% &=& - \int_{\Omega_k\setminus (B_R \cup \Gamma_0^k)}\div[(1-\zeta) \CC e(u_0^k)]. w_k\; dx, \notag
% \end{eqnarray}
% and because $\div(\CC e(u_0^k))=0$ in $\Omega_k\setminus \Gamma_0^k$ we obtain
$$ \int_{\Omega_k\setminus B_{R''}}(1-\zeta) \CC e(u_{\e_k}):e(w_k-r_k) \, dx  =  \int_{\Omega_k\setminus B_{R''}} (\nabla \zeta \odot   (w_k-r_k)) : \CC e(u_{\e_k}) \,dx .$$ 
Letting $R'' \nearrow R$ leads to
\begin{multline*}
\frac{1}{\e_k}\G(\e_k \Gamma_{\e_k}) =\frac{1}{2} \int_{\Omega_k}  \CC e(w_k-r_k):e(w_k-r_k)\, dx\\
+ \int_{B_{R'}} \zeta\; \CC e(u_{\e_k}):e(w_k-r_k)\, dx+  \int_{\Omega_k\setminus B_{R}}(\nabla \zeta \odot  ( w_k-r_k)) : \CC e(u_{\e_k}) \,dx .
\end{multline*}
Recalling from Theorem \ref{thm:blowup} that $u_{\e_k} \to u_{\Sigma_0}$ strongly in $L^2_{\rm loc}(\R^2;\R^2)$, and $e(u_{\e_k})\to e(u_{\Sigma_0})$ strongly in $L^2_{\rm loc}(\R^2;\Ms)$, while $w_k-r_k \to w$ strongly in $L^2_{\rm loc}(\R^2;\R^2)$, and $e(w_k-r_k)\wto e(w)$ weakly in $L^2(\R^2;\Ms)$, we infer that
\begin{multline}
 \liminf_{j\to \infty}\frac{1}{\e_k}\G(\e_k \Gamma_{\e_k}) \geq \frac{1}{2} \int_{\R^2}  \CC e(w):e(w)\, dx \\
 + \int_{B_{R'}} \zeta\; \CC e(u_{\Sigma_0}):e(w) \,dx +  \int_{\Omega_k\setminus B_{R}}(\nabla \zeta \odot   w) : \CC e(u_{\Sigma_0}) \,dx . \label{liminf}
 \end{multline}
We now let $\zeta$ be the Lipschitz and radial function defined by
\begin{equation}\label{eq:cut-off}
\zeta(x)=\left\{
\begin{array}{lll}
1 & \text{if} & x \in B_R,\\
\ds \frac{|x|-R}{R'-R} & \text{if} & x \in B_{R'}\setminus B_R,\\
0 & \text{if} & x \in \R^2 \setminus B_{R'}.
\end{array}
\right.
\end{equation}
% to $\frac{|x|-R}{R'-R}$ on $B_{R'}\setminus B_R$, equal to $0$ on $B_R$ and $1$ on $\R^2\setminus B_{R'}$. %By this way we obtain that  $\nabla \zeta(x) = \frac{1}{R'-R}\frac{x}{|x|}$ on $B_{R'}\setminus B_R$. 
Letting  $R'\to R$ in the right-hand side of \eqref{liminf} we finally get that, for $\LL^1$-a.e. $R>0$,
$$ \liminf_{j\to \infty}\frac{1}{\e_k}\G(\e_k \Gamma_{\e_k}) \geq \frac{1}{2} \int_{\R^2}  \CC e(w):e(w) \, dx
 + \int_{B_{R}}  \CC e(u_{\Sigma_0}):e(w) \,dx +  \int_{\partial B_R}  w \cdot (\CC e(u_{\Sigma_0}\nu)) \,d\HH^1 .$$

%%%%%%%%%%%%%%%%%%%%%
\medskip

\noindent {\bf Reduction to competitors in $H^1(\R^2 \setminus (\Sigma_0 \cup  \mathcal R(\Gamma));\R^2)$ with compact support. } In order to show that $w$ is a minimizer of the limit problem \eqref{eq:limit-pb}, we start by establishing that, without loss of generality, competitors in \eqref{eq:limit-pb} can be taken in $H^1(\R^2 \setminus (\Sigma_0 \cup  \mathcal R(\Gamma));\R^2)$ with compact support. First we reduce to the case where the competitor belong to $LD(\R^2 \setminus (\Sigma_0 \cup  \mathcal R(\Gamma)))$ have compact support. To this purpose, let us show that any $z\in LD(\R^2\setminus (\Sigma_0\cup  \mathcal R(\Gamma)))$ can be approximated strongly in $LD(\R^2\setminus (\Sigma_0\cup  \mathcal R(\Gamma)))$ by  functions with compact support. To this aim we consider $\varphi\in \C^\infty_c(B_2;[0,1])$ satisfying $\varphi=1$ on $B_1$, and define
$$\varphi_R(x):=\varphi\left(\frac{x}{R}\right).$$
We assume that $R$ is large enough so that $\Gamma\subset B_R$. Then we set $z_R:=(z-m_R)\varphi_R$ where $m_R$ is a suitable rigid movement associated to the Poincar\'e-Korn inequality in the domain $B_{2R}\setminus (B_R\cup \Sigma_0)$ (which is diffeomorphic to the Lipschitz set $(0,2\pi R) \times (0,R)$), namely
\begin{eqnarray}
\int_{B_{2R}\setminus (B_R\cup\Sigma_0)}|z-m_R|^2 \; dx \leq CR^2 \int_{B_{2R}\setminus (B_R\cup \Sigma_0)} |e(z)|^2 \; dx  \label{kornetoile}
\end{eqnarray}
%It is clear that $z_R$ converges strongly to $z$ in $L^2(\R^2)$. 
Moreover a immediate computation yields
$$e(z_R)=\varphi_R e(z)+\frac{1}{R}\nabla \varphi\left(\frac{\cdot}{R}\right)\odot (z-m_R).$$
The first term converges strongly to $e(z)$ in $L^2(\R^2;\Ms)$, while the second term converges to $0$ strongly in $L^2(\R^2;\Ms)$ due to \eqref{kornetoile}. As a consequence $z_R \to z$ strongly in $LD(\R^2 \setminus (\Sigma_0 \cup  \mathcal R(\Gamma)))$.

Next, we reduce to the case where $z$ lies in the Sobolev space $H^1(\R^2\setminus (\Sigma_0\cup  \mathcal R(\Gamma));\R^2)$. Let $D$ and $D'$ be bounded open sets such that
${\rm Supp}( z) \subset D' \subset \subset D$. According to the density result \cite[Theorem 1]{C}, we get the existence of a sequence $(z_n) \subset H^1(D\setminus (\Sigma_0\cup  \mathcal R(\Gamma));\R^2)$ such that $z_n \to z$ strongly in $L^2(D;\R^2)$ and $e(z_n)\to e(z)$ both strongly in $L^2(D;\Ms)$. This implies in particular that $z_n \to 0$ in $L^2(D\setminus D';\R^2)$. Let $\varphi \in \C^\infty_c(D;[0,1])$, $\varphi=1$ on $D'$, and set $\hat z_n=\varphi z_n \in H^1(\R^2\setminus (\Sigma_0\cup  \mathcal R(\Gamma)))$ with  ${\rm Supp}(\hat z_n) \subset D$, and satisfying $\hat z_n\to z$ strongly in $L^2(\R^2;\R^2)$, and $e(\hat z_n)\to e(z)$ strongly in $L^2(\R^2;\Ms)$.

%%%%%%%%%%%%%%%%%
\medskip

\noindent {\bf Upper bound and minimality.} We now assume that $z\in H^1(\R^2\setminus (\Sigma_0\cup  \mathcal R(\Gamma));\R^2)$ with compact support, contained in some bounded open set $D$. Clearly the number of connected components of $\partial D \cup ((\Sigma_{\e_k} \cup\mathcal R_{\e_k}(\Gamma_{\e_k})) \cap D)$ is bounded. Hence by \cite{BV} or \cite{CD} we get the existence of $z_k\in H^1(D \setminus (\Sigma_{\e_k} \cup \mathcal R_{\e_k}(\Gamma_{\e_k}));\R^2)$ such that $z_k \to z$ strongly in $L^2(D;\R^2)$ and $(\nabla z_k)\chi_{D\setminus (\Sigma_{\e_k} \cup \mathcal R_{\e_k}(\Gamma_{\e_k}))}\to \nabla z$ strongly in $L^2(D;\Ms)$. Multiplying by the same cut-off function $\varphi$ as in the previous step, we can also assume that $z_k=0$ in a neighborhood of $\partial D$. In this way we have obtained $z_k \in H^1(\R^2\setminus (\Sigma_{\e_k} \cup \mathcal R_{\e_k}(\Gamma_{\e_k}));\R^2)$ satisfying
$${\rm Supp}(z_k)\subset D\subset \Omega_k \quad \text{(for $k$ large enough)},$$
 $$z_k \to z \text{ strongly in }L^2(\R^2;\R^2),$$
$$(\nabla z_k)\chi_{\R^2\setminus (\Sigma_{\e_k} \cup \mathcal R_{\e_k}(\Gamma_{\e_k}))} \to \nabla z \text{ strongly in }L^2(\R^2;\Ms).$$
According to the minimality property of $w_k$ (see \eqref{JFF}), we have 
\begin{eqnarray}
\frac{1}{2}\int_{\O_k} \CC e(w_k) : e(w_k) \, dx &+&\int_{\O_k} \CC e(u_{\e_k}) : e(w_k) \,dx \notag \\
&\leq& \frac{1}{2}\int_{\O_k} \CC e(z_k) : e(z_k) \,dx +\int_{\O_k} \CC e(u_{\e_k}) : e(z_k) \,dx.
\end{eqnarray}
Let $\zeta$ be the cut-off function defined in \eqref{eq:cut-off}, then performing an integration by parts exactly as we did in step 3 (with $z_k$ instead of $w_k-r_k$) we arrive at the following
%We now let $\zeta$ be the same Lipschitz and radial function equal to $\frac{\|x\|-R}{R'-R}$ on $B_{R'}\setminus B_R$, equal to $0$ on $B_R$ and $1$ on $\R^2\setminus B_{R'}$ as  in Step 3, for asome $0<R<R'$. Writing $1=1-\zeta+\zeta$ and integrating by parts like we already did in Step 3 yields
\begin{multline*}
 \frac{1}{\e_k}\G(\e_k \Gamma_{\e_k}) =\frac{1}{2}\int_{\O_k} \CC e(w_k) : e(w_k) \,dx +\int_{\O_k} \zeta \CC e(u_{\e_k}) : e(w_k) \,dx +\int_{\O_k} [\nabla\zeta\odot w_k] :\CC e(u_{\e_k})\, dx \\
\leq \frac{1}{2}\int_{\O_k} \CC e(z_k) : e(z_k) \,dx +\int_{\O_k}\zeta \CC e(u_{\e_k}) : e(z_k) \,dx +\int_{\O_k} [\nabla\zeta\odot z_k] :\CC e(u_{\e_k}) \,dx .
\end{multline*}
The convergences established so far for the sequences $(z_k)$ and $(u_{\e_k})$ enable one to pass to the limit in the previous expression, first as $k \to \infty$ and then $R' \to R$. We finally get that
\begin{multline}
\limsup_{j\to \infty} \frac{1}{\e_k}\G(\e_k \Gamma_{\e_k}) \\
\leq \frac{1}{2}\int_{\R^2} \CC e(z) : e(z) \,dx +\int_{B_R}\CC e(u_{\Sigma_0}) : e(z) \, dx +\int_{\partial B_R} z \cdot (\CC e(u_{\Sigma_0})\nu) \,d\HH^1\label{limsupineq}
\end{multline}
for almost every $R>0$. By the density result established in step 4, inequality \eqref{limsupineq} holds for any $z\in LD(\R^2\setminus (\Sigma_0\cup \mathcal R(\Gamma)))$. Taking $z=w$, and gathering with \eqref{liminf} yields
%\begin{eqnarray}
%E_R(z):=\frac{1}{2}\int_{\R^2} \CC e(z) : e(z) \;dx +\int_{B_R}\CC e(u_{\Sigma_0}) : e(z) \;dx +\int_{\partial B_R} [\nu \odot z] :\CC e(u_{\Sigma_0}) \;dx. \notag
%\end{eqnarray}But   also  says that 
$$\lim_{j\to \infty} \frac{1}{\e_k}\G(\e_k \Gamma_{\e_k}) = \frac{1}{2}\int_{\R^2} \CC e(w) : e(w) \,dx +\int_{B_R}\CC e(u_{\Sigma_0}) : e(w) \, dx +\int_{\partial B_R} w \cdot (\CC e(u_{\Sigma_0})\nu) \,d\HH^1,$$
and using again \eqref{limsupineq}, we deduce that $w$ is a solution of the minimization problem \eqref{eq:limit-pb} for a.e. $R>0$ with $\Gamma \subset B_R$. 
%\begin{eqnarray}
%  \liminf_{k\to +\infty}\frac{1}{\e_k}\G_{\e_k}(\e_k \Gamma_{\e_k})\geq E_R(w).\label{liminfineq}
% \end{eqnarray} 
%  In particular, by taking $z=w$ in \eqref{limsupineq} and gathering together with \eqref{liminfineq}  we infer that
%$$ \lim_{k\to +\infty}\frac{1}{\e_k}\G_{\e_k}(\e_k \Gamma_{\e_k})=E_R(w),$$
%and returning to \eqref{limsupineq} again we deduce that $w$ is a minimizer. This ends the proof of the Theorem, except that we have obtain the conclusion for almost every $R>0$ and not all $R>0$. But this is not a serious issue: indeed, 
Finally, an integration by parts ensures that the value of $\mathcal{F}(\Gamma)$ is independent of $R>0$ and {\it a fortiori} holds for every $R>0$.
\end{proof}

\appendix

\section{Technical lemmas}

\noindent The object of this appendix is to prove several technical results used throughout this work. Let us recall few notations: $\Gamma_0 \in \K(\ol \O)$ is the original crack, and $B$ is an open ball centered at the origin such that $\ol{B} \subset \O$ and $\partial B \cap \Gamma_0\neq \emptyset$. In addition, $U$ is a smooth open set such that $\ol U \cap \Gamma_0= \emptyset$ and $U \cap \partial B \neq \emptyset$.

In the proof of Lemma \ref{lem:XY}, we used the following auxiliary result. 

\begin{lem}\label{lem:approx-norm-trace}
For any $g \in L^2(U \cap \partial B)$, there exists a function $\sigma \in L^2(U;\R^2)$ with $\div \sigma=0$ in $H^{-1}(U)$, $\sigma\nu=0$ in $H^{-1/2}(\partial U)$ and
$\sigma\nu=g \text{ in }L^2(U \cap \partial B)$.
\end{lem}

\begin{proof}
For any $u \in H_0^1(U)$, let 
$$T_1(u):=\int_{U \cap \partial B} gu \, d\HH^1.$$
The mapping $T_1:H_0^1(U) \to \R$ is clearly linear, and it is in addition continuous since by the trace theorem,
$$|T_1(u)|Ê\leq \|u\|_{L^2(U \cap \partial B)} \|g\|_{L^2(U \cap \partial B)} \leq C \|u\|_{H^1(U)}.$$
Therefore, $T_1 \in H^{-1}(U)$, and thus, there exists $\sigma_1 \in L^2(U;\R^2)$ such that
$$T_1(u)= \int_U \sigma_1 \cdot \nabla u\, dx \quad \text{ for any } u \in H_0^1(U).$$
Taking in particular $u \in H^1_0(U \cap B) \subset {\rm Ker}(T_1)$, we deduce by definition of weak derivatives that $\div \sigma_1=0$ in $H^{-1}(B \cap U)$, and, using the integration by parts formula in $H^1(U \cap B)$, that
$$\int_{U \cap \partial B} gu \, d\HH^1=T_1(u)=\langle \sigma_1 \nu , u \rangle_{H^{-1/2}(\partial(U \cap B)),H^{1/2}(\partial(U \cap B))}.$$
This shows that $\sigma_1 \nu =g$ in $L^2(U \cap \partial B)$ (where $\nu$ is the outer normal to $\partial B$), and $\sigma_1 \nu =0$ in $[H^{1/2}(B \cap \partial U)]'$. 

Arguing similarly on $U \setminus \ol B$, we  get that $\div \sigma_1=0$ in $H^{-1}(B \setminus \ol U)$, $\sigma_1 \nu =g$ in $L^2(U \cap \partial B)$ (where now $\nu$ is the inner normal to $\partial B$) and $\sigma_1 \nu =0$ in $[H^{1/2}(\partial U \setminus \ol B)]'$. Let us define $\sigma \in L^2(B;\R^2)$ by
$\sigma=\sigma_1$ in $U \cap B$ and $\sigma=-\sigma_1$ in $U \setminus B$. Clearly,  $\sigma \nu =0$ in $H^{-1/2}(\partial U)$, and since the normal trace of $\sigma$ do not jump across $\partial B \cap U$, we infer that $\div \sigma=0$ in $H^{-1}(U)$ and $\sigma \nu=g$ in $L^2(\partial B \cap U)$ (where $\nu$ is the outer normal to $\partial B$). 
\end{proof}

In the proof of Proposition \ref{PROPbound} and Theorem \ref{main2}, we used the following generalized integration by parts formula (see Lemmas 3.1 and 3.2 in \cite{CL} for a similar result in the scalar case).

\begin{lem}\label{lem:IPP}
Let $\Gamma \in \K(\overline \O)$. There exists a set $\mathcal N \subset \R^+$ of zero Lebesgue measure with the following property: for all $v \in H^1(\O \setminus (\Gamma_0 \cup \Gamma);\R^2)$ such that $v=0$ on $\partial \O \setminus \Gamma_0$, and for all $r \not\in \mathcal N$ with $\Gamma \subset B_r \subset\subset \O$, one has
$$\int_{(\O \setminus \Gamma_0) \setminus B_r} \sigma_0 : e(v)\, dx = -\int_{\partial B_r \setminus \Gamma_0}(\sigma_0 \nu) \cdot v\, d\HH^1.$$
%\langle \sigma_0 \nu, v \rangle_{[H^{1/2}(\partial B_r \setminus \Gamma_0)]',H^{1/2}(\partial B_r \setminus \Gamma_0)}.$$
\end{lem}

\begin{proof}
Let $r'<r$ be such that $\Gamma \subset B_{r'}Ê\subset \subset B_rÊ\subset \subset\O$, and consider the cut-off function $\eta \in W^{1,\infty}(\O:[0,1])$ given by
$$\eta(x):=\left\{
\begin{array}{cl}
1 & \text{ on } \O \setminus B_r,\\
0 & \text{ on } B_{r'},\\
\ds \frac{|x|-r'}{r-r'} & \text{ on } B_r \setminus B_{r'}.
\end{array}
\right.$$

We set $w:=\eta v$ so that $w \in H^1(\O \setminus \Gamma_0;\R^2)$ and $w=0$ on $\partial \O \setminus \Gamma_0$. Since $w=v$ outside $B_r$, we infer that
$$\int_{(\O \setminus \Gamma_0) \setminus B_r} \sigma_0 : e(v)\, dx =\int_{(\O \setminus \Gamma_0) \setminus B_r} \sigma_0 : e(w)\, dx
=\int_{\O \setminus \Gamma_0} \sigma_0 : e(w)\, dx-\int_{B_r \setminus \Gamma_0} \sigma_0 : e(w)\, dx.
$$
According to the variational formulation \eqref{eq:var-form}, we have
$$\int_{\O \setminus \Gamma_0} \sigma_0 : e(w)\, dx=0.$$
On the other hand, since $e(w)=\eta e(v) + \nabla \eta \odot v$ and $\nabla \eta(x)=\frac{1}{r-r'}\frac{x}{|x|} \chi_{B_r \setminus B_{r'}}$, we deduce that
$$\int_{(\O \setminus \Gamma_0) \setminus B_r} \sigma_0 : e(v)\, dx =\int_{B_r \setminus \Gamma_0} \eta \sigma_0 : e(v)\, dx - \frac{1}{r-r'}\int_{(B_r \setminus B_{r'}) \setminus \Gamma_0}\sigma_0 :\left( \frac{x}{|x|}\odot v\right)\, dx.$$
Letting $r' \to r$, we get that 
$$\int_{B_r \setminus \Gamma_0} \eta \sigma_0 : e(v)\, dx \to \int_{B_r \setminus \Gamma_0} \sigma_0 : e(v)\, dx,$$
while Lebesgue's differentiation theorem applied to the integrable function $\rho \mapsto\int_{\partial B_\rho \setminus \Gamma_0}(\sigma_0\nu ) \cdot v\, d\HH^1$ yields
$$ \frac{1}{r-r'}\int_{(B_r \setminus B_{r'}) \setminus \Gamma_0}\sigma_0 : \left(\frac{x}{|x|}\odot v\right) \, dx  \to \int_{\partial B_r \setminus \Gamma_0}(\sigma_0\nu )\cdot  v\, d\HH^1,$$
for all $r\not\in \mathcal N_v$, where $\mathcal N_v \subset \R^+$ is a measurable set of zero Lebesgue measure. The fact that the exceptional set can be chosen independently of the test function $v$ results from the separability of the space $\{v \in H^1(\O \setminus (\Gamma_0 \cup \Gamma);\R^2): \; v=0\text{ on }\partial \O \setminus \Gamma_0\}$.
\end{proof}

\section{A short review of Kondrat'ev theory}\label{K}

\noindent We follow the notations and statements of the book \cite[Section 6.1]{kmross} that we briefly recall here in the case of the bilaplacian in the cracked plane $\R^2\setminus \Sigma_0$.  Let us consider weak solutions of the problem
$$
(P_1)\qquad \left\{
\begin{array}{ll}
\Delta^2 w=f  & \text{ in } \R^2\setminus \Sigma_0, \\
w =0 \text{ and } \frac{\partial w}{\partial \nu}=0 & \text{ on  } \Sigma_0,
\end{array}
\right.
$$
in weighted Sobolev spaces of type $V_\beta^\ell(\R^2 \setminus \Sigma_0)$ (see the definition in Section \ref{sectionKond}) which is the core of Kondrat'ev's Theory. It is easily seen that $\Delta^2$ (associated with homogenous Dirichlet conditions) maps $w \in V_\beta^\ell(\R^2 \setminus \Sigma_0)$ to $f \in V_\beta^{\ell-4}(\R^2 \setminus \Sigma_0)$. For $\ell\geq 4$ this fact is quite obvious from the  definition, and for $\ell<4$, it follows from a standard extension argument (see \cite[Theorem 6.1.2]{kmross}). Kondrat'ev theory ensures that this operator is actually of Fredhlom type, and that it defines an isomorphism provided $\beta \in\R \setminus \mathcal{S}$ and $\ell \in \mathbb{Z}$, where $\mathcal{S}$ is an exceptional countable set.
%Before giving an exact statement, let us say a few word about this exceptional set $\mathcal{S}$.
In our special case it turns out  to be contained in the set of half integers  $\frac12 \mathbb{Z}$, as for most  elliptic operators (see \cite{cosdau}). Indeed, this set appears as  the \emph{spectrum} of the Mellin transform of the operator written in polar coordinates, with corresponding boundary conditions. In the language of \cite{kmross} this will be called the Pencil operator, denoted by $\mathfrak{A}(\lambda)$ and studied in \cite[Chapter 5]{kmross} (and defined pp. 197 in \cite{kmross} in the case that we are interested in).  
The exact computations in the special case of the bilaplacian are quite standard, and can be found for instance in \cite[Chapter 7.1]{kmr2} (see also \cite[Section 7.2.1]{grisvard}, but with different notations and conventions leading to slightly different characteristic equations). Let us recall here those computations, still using the language of \cite{kmross}. 

First we recall that the Mellin transform of a function $g \in \C^\infty_c(\R^+)$ is given by
\begin{eqnarray}
\hat u(\lambda)=\int_0^{+\infty} r^{-\lambda-1}g(r)dr, \quad \text{ for all } \lambda \in \mathbb{C}. \label{Mellin}
\end{eqnarray}
Another way to understand this transformation is by taking the Laplace transform of the function $t \mapsto g(e^t)$. Relevant properties are recalled in \cite[Lemma 6.1.3]{kmross}, and one of the most important is probably 
\begin{eqnarray}
\widehat{r\partial_r g} = \lambda \hat g \label{mellin1}.
\end{eqnarray}
Now let us look for the pencil operator. Since it is obtained via the Mellin transform of $\Delta^2$ (up to a factor $r^4$), we need to write it in polar coordinates $(r, \theta)$ which gives  
$$\Delta^2 =\partial_r^4+\frac{2\partial_r^3}{r}-\frac{\partial_r^2}{r^2}+\frac{\partial_r}{r^3}  + \frac{\partial_\theta^4}{r^4}+\frac{4\partial_\theta^2}{r^4}-\frac{2\partial_\theta^2\partial_r }{r^3}+ \frac{2\partial_\theta^2\partial_r^2}{r^2}\; .$$
We then identify the terms of the form $(r\partial_r)^k$, and for this purpose we shall use the following elementary formulas
{\begin{eqnarray}
(r\partial_r)^2&=&r\partial_r + r^2\partial_r^2 \notag \\
(r\partial_r)^3&=&r\partial_r+3r^2\partial_r^2+r^3\partial_r^3 \notag \\
(r\partial_r)^4&=&r\partial_r+7r^2\partial_r^2+6r^3\partial_r^3+r^4\partial_r^4 \notag
\end{eqnarray}
which imply
\begin{eqnarray} 
\Delta^2 &=& r^{-4}\left([(r\partial_r)^4-4(r\partial_r)^3+4(r\partial_r)^2]+[2(r\partial_r)^2-4r\partial_r+4]\partial_\theta^2+\partial_\theta^4\right) \notag \\
&=:&r^{-4}\mathcal{L}(\partial_\theta,r\partial_r) \notag
\end{eqnarray}

The pencil operator $\mathfrak{A}(\lambda)$ is then obtained by taking the Mellin transform \eqref{Mellin} in the $r$ variable of the operator  $\mathcal{L}(\partial_\theta,r\partial_r)$ defined above. Using \eqref{mellin1} we therefore obtain
\begin{eqnarray}
\mathfrak{A}(\lambda)&=&(\lambda^4-4\lambda^3+4\lambda^2)+(2\lambda^2-4\lambda+4)\partial_\theta^2+\partial_\theta^4\notag \\
&=&(\partial^2_\theta+(\lambda-2)^2)(\partial^2_\theta+\lambda^2),
\end{eqnarray}
%(this corresponds to (7.2.1.10) in \cite{grisvard}, but with $\lambda=i\tau$), 
and the boundary conditions in the variable $\theta$ are still zero ({\it i.e.} acting on functions $\varphi$ with the boundary conditions $\varphi(0)=\varphi(2\pi)=\varphi'(0)=\varphi'(2\pi)=0$). The set $\mathcal{S}$ is then the \emph{spectrum} of $\mathfrak{A}(\lambda)$, and according to the terminology of Operator Pencils this means the set of  $\lambda$ for which the operator is non invertible \cite[Chapter 5]{kmross}. %Since $\mathfrak{A}(\lambda)=\mathfrak{A}(2-\lambda)$, it is enough to find the values of $\lambda$ in the region $Re(\lambda)\geq 1$. 
By \cite[Chapter 7.1]{kmr2} (see in particular the last paragraph  before Section 7.2 for the special case $\alpha=2\pi$), this set is real and
$$\mathcal{S}=\left\{1\pm \frac{k}{2} \, ; \; k \in \N\setminus \{0\}\right\}.$$
All of them, except $\lambda=0$ and $\lambda=2$, have geometric and algebraic multiplicities equal to $2$. The associated eigenfunctions are given by explicit functions that one can find in \cite[formulas (7.1.14) and (7.1.15)]{kmr2}. We shall only give the ones corresponding to $\lambda=3/2$, which are the functions defined in \eqref{defu1} and \eqref{defu2}.

According to all the above facts, a direct application of \cite[Theorem 6.1.3]{kmross} yields

\begin{thm}\label{kondth0} 
If $\beta\in \R$ and $\ell \in \mathbb{Z}$ are such that 
$$-\beta+\ell-1 \not\in \mathcal{S},$$
then for every $f\in V_\beta^{\ell-4}(\R^2 \setminus \Sigma_0)$, there exists a unique solution $w \in V_\beta^\ell(\R^2 \setminus \Sigma_0)$ of $(P_1)$.
\end{thm}

In addition, a direct application of \cite[Theorem 6.1.5]{kmross} implies that
\begin{thm}\label{kondth} 
Let $\beta_2<\beta_1$ be two real numbers, $\ell\in \mathbb{Z}$, and assume that 
$$-\beta_i+\ell-1\not\in \mathcal{S}, \quad \text{ for all }  i\in \{1,2\}.$$
If $w \in V_{\beta_1}^\ell(\R^2 \setminus \Sigma_0)$ is a solution of $(P_1)$ with $f\in  V_{\beta_1}^{\ell-4}(\R^2 \setminus \Sigma_0) \cap V_{\beta_2}^{\ell-4}(\R^2 \setminus \Sigma_0)$,
then there exists $z \in V_{\beta_2}^\ell(\R^2 \setminus \Sigma_0)$ such that 
$$w-z = \sum_{\lambda \in \mathcal{S}\cap (1-\beta_1,  1-\beta_2) } r^{\lambda} \varphi_\lambda(\theta),$$
where the $\varphi_\lambda$ are linear combinations of eigenfunctions of $\mathfrak{A}(\lambda)$. In particular $\varphi_{3/2}=c_1\psi_1+c_2\psi_2$ where $\psi_1$ and $\psi_2$ are defined in \eqref{defu1} and \eqref{defu2}.
\end{thm}

%Actually, Theorem 1.2. of \cite{kond67} is much more general and is stated for any elliptic operator of any order, in any dimension and for general  domains of arbitrary aperture. We stated a particular case in dimension 2 for the bi-Laplacian in fractured domain. For this particular case the functions $\phi_k(\theta)$ are  linear combination of explicit functions (but we shall not give them explicitly here), and we also used the knowledge of the singular set denoted here by $\mathcal{S}$ which is of the form $\{ \frac{k}{2}\; ; \; k \in \mathbb{Z}\}$ (see for instance \cite{cosdau} or \cite{grisvard}). In general this set may be complex, and contains the poles of a  meromorphic function valued in a certain functional space linked with the Melin transform of the operator, but here it appears to be real and containing only the half-integers, which is most often the case in fractured domains as proved in \cite{cosdau}.

%Thanks to Theorems \ref{kondth0} and \ref{kondth}, we are now in position to prove Proposition \ref{eq:3/2homo}.

\noindent {\bf Aknowledgements.} The authors wish to thank Svitlana Mayboroda for useful discussions about the subject of this paper, and for having pointed out reference \cite{KKLO}. They are also grateful to Monique Dauge for having sent them a copy of the paper \cite{kond67}, and for the argument leading to the proof of Proposition \ref{eq:3/2homo}. J.-F. Babadjian has been supported by the {\sl Agence Nationale de la Recherche} under Grant No. ANR 10-JCJC 0106.   A. Chambolle and A. Lemenant has been partially supported by the {\sl Agence Nationale de la Recherche} under Grant No. ANR-12-BS01-0014-01 GEOMETRYA.

%%%%%%%%%%%%%%%%%%%%%%%%%%%%%%%%%%%

\end{document}